\documentclass{amsart}
\usepackage{amsmath}
\usepackage{amsthm}
\usepackage{amssymb}
\usepackage{url}
\usepackage{cite}
\usepackage{bm}
\usepackage{verbatim}
\usepackage{indentfirst}
\usepackage{hyperref}
\usepackage[capitalize]{cleveref}
\hypersetup{hidelinks}
\newcommand{\norm}[1]{\lVert#1 \rVert}

\title{A compactness theorem for hyperk\"ahler 4-manifolds with boundary}
\author{Hongyi Liu}
\newtheorem{theorem}{Theorem}[section]
\newtheorem{lemma}[theorem]{Lemma}
\newtheorem{definition}[theorem]{Definition}
\newtheorem{question}[theorem]{Question}
\newtheorem{proposition}[theorem]{Proposition}
\newtheorem{corollary}[theorem]{Corollary}

\theoremstyle{definition}
\newtheorem{remark}[theorem]{Remark}

\begin{document}
\begin{abstract}
In this paper, we study the compactness of a boundary value problem for hyperk\"ahler 4-manifolds. We show that under certain topological conditions and the positive mean curvature condition on the boundary, a sequence of hyperk\"ahler triples converges smoothly up to diffeomorphisms if and only if their restrictions to the boundary converge smoothly up to diffeomorphisms. We also generalize this result to torsion-free hypersymplectic triples. 
\end{abstract}
\maketitle
\tableofcontents

\section{Introduction}
\indent A Riemannian metric $g$ on a $4$-manifold is called  hyperk\"ahler if its holonomy group  $Hol(g)$ is contained in $Sp(1)=SU(2)$. A closed hyperk\"ahler 4-manifold is diffeomorphic to either a torus or the K3 manifold, and the moduli space of all hyperk\"ahler metrics are described by Torelli theorems. There have been extensive recent studies on the Gromov-Hausdorff compactification of these moduli spaces, see for example \cite{odaka2018collapsing} \cite{sun2021collapsing}.

Hyperk\"ahler metrics in dimension 4 are the simplest models for Riemannian metrics with special holonomy. Little general existence theory is developed for the latter in dimensions greater than $4$, except for Calabi-Yau manifolds. Recently Donaldson \cite{donaldson2018elliptic} proposes to study special holonomy metrics on manifolds with boundary and set up suitable elliptic boundary value problems. To make further progress in this direction, it is clear that we need a compactness theory.

In this paper, we study the boundary value problem for hyperk\"ahler 4-manifolds, which serves as the first step towards Donaldson's program.  We follow the general set-up by Fine-Lotay-Singer \cite{fine2017space} in terms of \emph{hyperk\"ahler triples}. 
A hyperk\"ahler triple on an oriented smooth 4-manifold $X$ is a triple of symplectic forms $\bm{\omega}=(\omega_1, \omega_2, \omega_3)$  satisfying the following pointwise condition$$\omega_i\wedge\omega_j=\frac{1}{3}\delta_{ij}(\omega_1^2+\omega_2^2+\omega_3^2).$$
It is well-known that a hyperk\"ahler triple $\bm\omega$ uniquely determines a compatible hyperk\"ahler metric $g_{\bm\omega}$ such that for each $i$, $\omega_i^2=2\text{dvol}_{g_{\bm\omega}}$ and $\omega_i$ is parallel with respect to the Levi-Civita connection. Conversely, given a hyperk\"ahler metric $g$ on $X$, one can choose an orientation and find a compatible hyperk\"ahler triple $\bm\omega$, which is unique up to a $SO(3)$ rotation.

Now let $X$ be a compact oriented  smooth 4-manifold with boundary $\partial X$.  Note $\partial X$ has an induced orientation defined by contracting a volume form of $X$ with an outward vector field. If $\bm\omega$ is a hyperk\"ahler triple on $X$, then its restriction to $\partial X$  is a closed framing $\bm\gamma$ on $\partial X$. The following is a natural filling problem, proposed by \cite{fine2017space}.

\begin{question}\label{question filling}
Which closed framing $\bm\gamma$ extends to a hyperk\"ahler triple on $X$?
\end{question}

Notice a framing $\bm\gamma$ defines a Riemannian metric $g_{\bm\gamma}$ on $\partial X$ as follows: first, there exists a unique dual coframe $\bm\eta=(\eta_1,\eta_2,\eta_3)$ such that $\gamma_i=\frac{1}{2}\delta^{ijk}\eta_j\wedge\eta_k$ and such that  $\eta_1\wedge\eta_2\wedge\eta_3$ is compatible with the orientation of $\partial X$; then the Riemannian metric $g_{\bm\gamma}$ is defined by setting $\bm\eta$ to be orthonormal. When there is no ambiguity, we always use $\bm\eta$ to denote the dual coframe of $\bm\gamma$ defined in this way and denote the Hodge star operator of the Riemannian metric by $*_{\bm\gamma}=*_{\bm\eta}$. It is well-known that if $\bm\omega$ is a hyperk\"ahler triple, then $g_{\bm\omega}|_{\partial X}=g_{\bm\gamma}$; more importantly that the second fundamental form of $\partial X$ is determined intrinsically by $\bm\gamma$ via the  
matrix  $*_{\bm\eta}(\eta_i\wedge d\eta_j)$. In particular, the mean curvature  $H_{\bm\gamma}$ is given by one half of the trace of this matrix, i.e.,  $H_{\bm\gamma}=\frac{1}{2}*_{\bm\eta}(\bm\eta\wedge d\bm\eta^T)$. 

There are some previous works on Question \ref{question filling}.    Bryant \cite{MR2681703} studied the local ``thickening'' problem and obtained both positive and negative results.   It was shown that any real analytic closed framing on a closed oriented 3-manifold $Y$ can be extended to a hyperk\"ahler triple on  $Y\times (-\epsilon,\epsilon)$ for some $\epsilon>0$, and the extension is essentially unique. On the other hand, there exists a smooth closed framing on an open ball  $B^3
\subset \mathbb{R}^3$ that cannot be extended to a hyperk\"ahler triple on $B^3\times (-\epsilon,\epsilon)$ for any $\epsilon>0$. Fine-Lotay-Singer\cite{fine2017space} studied the local deformation theory for Question \ref{question filling} and showed that the boundary framings must deform in certain directions. Roughly speaking, let $X=B^4$ for simplicity, suppose $\bm\omega$ is a hyperk\"ahler triple such that $\partial X$ has positive mean curvature, and $\bm\omega'$ is a nearby hyperk\"ahler triple,
 then after moduling out diffeomorphisms of $\partial X$, the dual coframe of $\bm\omega'|_{\partial X}$ must be a small pertubation of that of $\bm\omega|_{\partial X}$ in the direction of negative frenquency of the boundary Dirac operator defined by $g_{\bm\omega}|_{\partial X}$.

A sequence of pairs of smooth covariant tensors $(T_i^1,\cdots, T_i^m)$ on a compact manifold $M$ with empty or nonempty boundary is said to converge in \emph{Cheeger-Gromov} sense to $(T^1,\cdots, T^m)$ on $M$, if there exist diffeomorphisms $f_i:M\rightarrow M$ such that $f_i^*T_i^{1}\rightarrow T^1,\cdots, f_i^*T_i^m\rightarrow T^m $ smoothly on $M$.

Our main result is the following closedness result for Question \ref{question filling} :

\begin{theorem}\label{convergence of triples}
Let $X$ be a compact oriented smooth 4-manifold with boundary, such that there does not exist  $C\in H_2(X,\mathbb{Z})$ with self intersection $C^2=-2$. Let $\bm\omega_i$ be a sequence of smooth hyperk\"ahler triples on $X$. Suppose $\bm\omega_i|_{\partial X}$ converges in Cheeger-Gromov sense to a closed framing $\bm\gamma$ on $\partial X$ such that $H_{\bm\gamma}>0$, then there exists a smooth hyperk\"ahler triple $\bm\omega$ on $X$ with $\bm\omega|_{\partial X}=\bm\gamma$ and $\bm\omega_i$ converges in Cheeger-Gromov sense to $\bm\omega$ on $X$.

\end{theorem}
The proof here includes two parts: the compactness and uniqueness. The former is the main story of this paper, and the latter is a consequence of \cite{biquard:hal-02928859} or \cite{anderson2008unique} on unique continuation of Einstein metrics with prescribed boundary metric and second fundamental form. It is worth noting that for the compactness part, no general Riemannian convergence theory can be applied directly. The difficulty here is that we only have data on the boundary, and a priori we do not know anything near the boundary or in the interior. Specifically, we worry about the following three bad geometric behaviours:  curvature blow up, volume collapsing and boundary touching. These things are entangled, making it difficult to rule out any of them. However, we are able to separate these bad behaviours and rule them out. We will also give examples to demonstrate that the assumptions in Theorem \ref{convergence of triples} are essential, see Remark \ref{positive mean curvature essential} and \ref{no -2 curve essential}.

Such $C$ in the assumption of Theorem \ref{convergence of triples} is usually called a ``\emph{$-2$ curve}'' in $X$, which appears in Kronheimer's classification of hyperk\"ahler ALE spaces \cite{kronheimer1989construction} \cite{kronheimer1989torelli}. They appear in bubble limits of volume-noncollapsed hyperk\"ahler manifolds. From this, one can replace the ``no $-2$ curve'' condition by an assumption on enhancements of $\bm\omega_i|_{\partial X}$. Let $\bm\gamma$ be a closed framing on $\partial X$. Following \cite{Donaldson2017BoundaryVP}\cite{donaldson2018elliptic}, an enhancement of $\bm\gamma$ is an equivalent class in the set of triples of  closed 2-forms on $X$ whose restrictions to $\partial X$ are equal to $\bm\gamma$. The equivalence relation is defined by $\bm\theta\sim \bm\theta+d\bm a$, where $\bm a$ is a triple of smooth 1-forms on $X$ vanishing on $\partial X$. From the de Rham cohomology exact sequence of the pair $(X,\partial X)$, 
$$H^2(X,\partial X)\rightarrow H^2(X)\rightarrow H^2(\partial X)\rightarrow H^1(X,\partial X),$$
we know $\bm\gamma$ has at least one enhancement if and only if each $\gamma_i$ lies in the kernel of $H^2(\partial X)\rightarrow H^1(X,\partial X)$, and we know the set of all enhancements of $\bm\gamma$ is an affine space over $H^2(X,\partial X)\otimes\mathbb{R}^3$. Choose an enhancement of $\bm\gamma$ and denote it by $\hat{\bm\gamma}$.  Given a 2-cycle $\Sigma\in H_2(X,\mathbb{Z})$, then for any triple of closed 2-forms $\bm\theta\in\hat{\bm\gamma}$, $\int_{\Sigma}\bm\theta$ does not depend on the choice of $\bm\theta$ and we denote this invariant by $c_{\hat{\bm\gamma},\Sigma}\in\mathbb{R}^3$.

The proof of Theorem \ref{convergence of triples} easily adapts to 
\begin{theorem}\label{convergence of triples enhancements}
Let $X$ be a compact oriented smooth 4-manifold with boundary. Let $\bm\omega_i$ be a sequence of smooth hyperk\"ahler triples on $X$, and $\hat{\bm\gamma}_i$ be the enhancement of $\bm\gamma_i=\bm\omega_i|_{\partial X}$ where $\bm\omega_i$ lie in. Let $a>0$ be a positive number. Suppose for any $C\in H_2(X,\mathbb{Z})$ with self intersection $C^2=-2$, $|c_{\hat\gamma_i,C}|\geq a$ and $\bm\omega_i|_{\partial X}$ converges in Cheeger-Gromov sense to a closed framing $\bm\gamma$ on $\partial X$ such that $H_{\bm\gamma}>0$.
Then there exists a smooth hyperk\"ahler triple $\bm\omega$ on $X$ with $\bm\omega|_{\partial X}=\bm\gamma$, and $\bm\omega_i$ converges in Cheeger-Gromov sense to $\bm\omega$ on $X$.
\end{theorem}

It is worth noting that Question \ref{question filling} is not an elliptic boundary value problem, observed by \cite{fine2017space}. This can also be seen from the uniqueness result of \cite{biquard:hal-02928859} or \cite{anderson2008unique}: the restriction of $\bm\omega$ to any open boundary portion determines $g_{\bm\omega}$ in the whole interior up to local isometries. So, it is natural to enlarge the class of closed triples of 2-forms on $X$ to obtain an elliptic boundary value problem. In \cite{donaldson2018elliptic}, Donaldson studied the deformation theory of torsion-free $G_2$ structures on a compact oriented 7-manifold with boundary $M^7$ as follows. Suppose $\phi_0$ is a smooth torsion-free $G_2$ structure, $\rho_0=\phi_0|_{\partial M^7}$, and denote the enhancement (defined in an analogous way as before) of $\rho_0$ where $\phi_0$ lies in by $\hat\rho_0$.  Donaldson set up an elliptic boundary value problem for the torsion-free equation. So in particular, if the kernel space at $\phi_0$ is trivial, then for any small closed 3-form $\theta$ on $X$,  $\hat\rho_0+\theta|_{\partial X}$ contains a unique torsion-free $G_2$ structure that is close to $\phi_0+\theta$ after gauge fixing. This phenomenon is quite different from the hyperk\"ahler case, since in that case we cannot deform the boundary framing arbitrarily to extend it to a hyperk\"ahler triple, as we discussed before.

It is well-known that $G_2$ structures have a reduction to dimension 4. Consider $X^4\times T^3$. A triple of two forms $\bm\omega=(\omega_1,\omega_2,\omega_3)$ on $X^4$  defines a 3-form on $X^4\times T^3$ by
\begin{equation}\label{g2 and hypersymplectic}
 \phi=dt^1\wedge dt^2\wedge dt^3-\omega_1\wedge dt^1-\omega_2\wedge dt^2-\omega_3\wedge dt^3.
\end{equation}
 The triple $\bm\omega$ is called torsion-free hypersymplectic if $\phi$ is a torsion-free $G_2$ structure. Locally, this is a weaker condition than being hyperk\"ahler, see examples in \cite{Donaldson2017BoundaryVP} or \cite{fine2020report}. Donaldson observed in  \cite{Donaldson2017BoundaryVP} that the boundary value problem for torsion-free $G_2$ structures can also be reduced to dimension 4. So, a compactness result is helpful to solve this dimension reduced boundary value problem. 

Similar to the hyperk\"ahler case, a torsion-free hypersymplectic triple $\bm\omega$ defines a Riemannian metric $g_{\bm\omega}$ and a positive definite $SL(3,\mathbb R)$-valued function $\bm Q=(Q_{ij})$ such that 
$$\omega_i\wedge \omega_j=2Q_{ij}\text{dvol}_{g_{\bm\omega}}$$
We denote $\bm Q'$ the restriction of $\bm Q$ to $\partial X$. When there is ambiguity, we use notations $\bm Q_{\bm\omega}$, $\bm Q'_{\bm\omega}$ to denote their dependence on $\bm\omega$. One can show that the mean curvature of $\partial X$ has an explicit expression in terms of $\bm\gamma$, $\bm Q'$, and we denote this explicit expression by $H_{\bm\gamma,\bm Q'}$. Note that on $\partial X$,  $\bm\gamma,\bm Q'$ are subject to the constraints $d\bm\gamma=0,  d(\bm\gamma(\bm Q')^{-1})=0$.

We have the following analogue of Theorem \ref{convergence of triples enhancements}:

\begin{theorem}\label{convergence of hypersymplectic triples}
Let $X$ be a compact oriented smooth 4-manifold with boundary. Let $\bm\omega_i$ be a sequence of smooth torsion-free hypersymplectic triples on $X$, and $\hat{\bm\gamma}_i$ be the enhancement of $\bm\gamma_i=\bm\omega_i|_{\partial X}$ where $\bm\omega_i$ lie in.  Let $a>0$ be a positive number. Suppose for any $C\in H_2(X,\mathbb{Z})$ with self intersection $C^2=-2$, $|c_{\hat\gamma_i,C}|\geq a$, and $(\bm\gamma_i,\bm Q_i')$  converges in Cheeger-Gromov sense to some pair $(\bm\gamma,\bm Q')$ on $\partial X$, such that $\bm\gamma$ is a framing, $\bm Q'$ is positive definite and $H_{\bm\gamma,\bm Q'}>0$. Then there exists a smooth torsion-free hypersymplectic triple $\bm\omega$ on $X$ with $\bm\omega|_{\partial X}=\bm\gamma$, $\bm Q_{\bm\omega}'=\bm Q'$, and $\bm\omega_i$
converges in Cheeger-Gromov sense to $\bm\omega$.

\end{theorem}
 
Note that Theorem \ref{convergence of hypersymplectic triples} includes the previous two versions.

This paper is organized as follows: In Section \ref{section triples}, we discuss basics of hyperk\"ahler triples on manifolds with boundary. In Section \ref{section Riemannian geometry}, we discuss Riemannian geometry for manifolds with boundary and Riemannian convergence theory. In Section \ref{section main proof}, we prove Theorem \ref{convergence of triples}  and Theorem \ref{convergence of triples enhancements} and give some remarks about the proofs. In Section \ref{section torsion-free triples}, we discuss some basics for torsion-free hypersymplectic triples and prove Theorem \ref{convergence of hypersymplectic triples}.

\textbf{Notations.}
$$\mathbb{R}_+^n=\{x\in\mathbb{R}^n:x^n\geq 0\},$$
 $$B_r=\{x\in\mathbb{R}^n: |x|<r\},$$
 $$B_r^+=B_r\cap \mathbb{R}_+^n,$$
 $$\tilde B_r=B_r^+\cap \partial\mathbb{R}_+^n,$$
$$\partial \tilde{B}_r=\{x\in \mathbb{R}^n:|x|=r,x^n=0\},$$
$$\partial^+ {B}_r^+=\{x\in \mathbb{R}^n:|x|=r,x^n>0\}.$$

\textbf{Acknowledgements.} The author is very grateful to his thesis advisor Song Sun for suggesting the problem, constant support and many inspiring discussions. I thank Antoine Song and Chengjian Yao for some useful comments. I thank Simon Donaldson and Jason Lotay for their interest in this work. 

The author is supported by Simons Collaboration Grant on Special Holonomy in Geometry, Analysis, and Physics (488633, S.S.).

\section{Hyperk\"ahler triples and closed framings}\label{section triples}

The discussions in this section are well-known.

\subsection{Pointwise theory}
 Let $V$ be an oriented 4-dimensional vector space, and ${\bm\omega}=(\omega_1,\omega_2,\omega_3)\in\Lambda^2(V^*)\otimes \mathbb{R}^3$. Suppose $\bm\omega$ is a \emph{definite} triple, i.e., 
$\omega_i$ spans a maximum positive subspace of $\Lambda^2(V^*)$ with respect to the wedge product, then $\omega_i$ defines a unique conformal structure on $V$ by making each $\omega_i$ self dual. Fix a volume form $\mu_0$ on $V$ that defines the orientation of $V$, write $\omega_i\wedge\omega_j=2q_{ij}\mu_0$, we define a matrix $\bm Q$ associated to the definite triple $\bm\omega$ by $Q_{ij}=\frac{q_{ij}}{\det(q_{ij})^\frac{1}{3}}$, which does not depend on the choice of $\mu_0$. We use $\bm Q^{-1}= (Q^{ij})$ to denote the inverse matrix of $\bm Q$. If we write $$\omega_i\wedge\omega_j=2 Q_{ij}\mu,$$ then $\mu$ is a volume form intrinsically defined by $\bm\omega$. We define a unique metric $\langle,\rangle_{\bm\omega}$ on $V$ in the conformal structure by making $\mu$ the volume form. Explicitly, $$\langle u,v\rangle_{\bm\omega}=\frac{1}{6}\sum_{i,j,k=1}^3 \delta^{ijk}\frac{\iota_u \omega_i\wedge\iota_v\omega_j\wedge\omega_k}{\mu}.$$
 So $$\langle u,u \rangle_{\bm\omega}=\frac{\iota_u\omega_1\wedge\iota_v\omega_2\wedge\omega_3}{\mu}.$$
Denote $*_{\bm\omega}$ the Hodge star operator defined by this metric.

Let $W$ be an orientated 3-dimensional  vector space, and ${\bm\gamma}=(\gamma_1,\gamma_2,\gamma_3)$ be a framing on $W$, i.e., a basis for $\Lambda^2W^*$. Then by elementary linear algerba, there exists a coframe  ${\bm\eta}=(\eta_1,\eta_2,\eta_3)$ such that

\begin{equation}\label{framing and coframe}
\gamma_i=\frac{1}{2}\delta^{ijk}\eta_j\wedge\eta_k.
\end{equation}
Such $\bm\eta$ is uniquely determined up to a sign and we choose $\bm\eta$ such that $\eta_1\wedge\eta_2\wedge\eta_3$ defines the orientation of $W$ and we denote this volume form by $\text{vol}_{\bm\gamma}$. There is a unique metric on $W$, denoted by $\langle,\rangle_{\bm\gamma}$, that makes $\bm\eta$ an orthonormal coframe. Denote $*_{\bm\gamma}=*_{\bm\eta}$ by the Hodge star operator of $\langle,\rangle_{\bm\gamma}$. Let $e_i\in W$ be the dual vector of $\eta_i$, so $\eta_i(e_j)=\delta_{ij}$.  Then $\bm e=(e_1,e_2,e_3)$ is a frame of $W$. Conversely, given a coframe $\bm\eta=(\eta_1,\eta_2,\eta_3)$ on $W$ compatible with the orientation, one can define a framing $\bm\gamma=(\gamma_1,\gamma_2,\gamma_3)$ via (\ref{framing and coframe}), and a volume form $\text{vol}_{\bm\eta}=\text{vol}_{\bm\gamma}$, a metric $\langle,\rangle_{\bm\eta}=\langle,\rangle_{\bm\gamma}$, a Hodge star opertaor $*_{\bm\eta}=*_{\bm\gamma}$.

Now if $W\subset V$ is a 3-dimensional subspace, $\bm\omega=(\omega_1,\omega_2,\omega_3)$ is a definite triple on $V$,   and $\bm\gamma=(\gamma_1,\gamma_2,\gamma_3)$ is the restriction of $\bm\omega$ to $W$.  Let $\langle,\rangle_{W}$ be the restriction of the metric $\langle,\rangle_{\bm\omega}$ on $W$, which defines a volume form $\text{vol}_W$ and a Hodge star operator $*_W$  compatible with the orientation of $W$.  Since $\omega_i$ are self-dual, we can write $$\omega_i=\nu^*\wedge *_W\gamma_i+\gamma_i,$$ where $\nu^*=*_{\bm\omega}\text{vol}_W$, then we have $$\omega_i\wedge\omega_j=2\nu^*\wedge*_W\gamma_i\wedge\gamma_j=2\langle \gamma_i,\gamma_j\rangle_{W}\nu^*\wedge \text{vol}_{W}=2\langle\gamma_i,\gamma_j\rangle_W\mu.$$ Hence on $W$, $$Q_{ij}=\langle\gamma_i,\gamma_j\rangle_W,$$ Since $\det(\bm Q)=1$, we have $\text{vol}_{\bm\gamma}=\text{vol}_{W}$. If furthermore we assume $Q_{ij}=\delta_{ij}$, then $\langle\gamma_i,\gamma_j\rangle_{W}=\delta_{ij}=\langle \gamma_i,\gamma_j\rangle_{\bm\gamma}$. In this case, $\langle,\rangle_{\bm\gamma}=\langle,\rangle_W$ as an inner product on $W=\Lambda^1 W$, so in particular $*_W=*_{\bm\gamma}$.

\subsection{Local theory}
Now we move our pointwise discussions to manifolds. Let $X$ be a oriented 4-manifold with boundary, $\bm\omega=(\omega_1,\omega_2,\omega_3)$ be a smooth section in $\Gamma(X, \Lambda^2T^*X\otimes\mathbb{R}^3)$ such that it is a definite triple pointwise, and $\bm\gamma=(\gamma_1,\gamma_2,\gamma_3) $ be its restriction to $\partial X$. By the discussions above, $\bm\omega$ defines a matrix valued function $\bm Q=(Q_{ij})$, a volume form $\mu$, and a Riemannian metric $g_{\bm\omega}$ which equals to  $\langle,\rangle_{\bm\omega}$ on the tangent spaces at each point. Similarly, $\bm\gamma$ defines $\bm\eta=(\eta_1,\eta_2,\eta_3)\in \Omega^1(\partial X)\otimes\mathbb{R}^3,\bm e=(e_1,e_2,e_3)\in \Gamma(\partial X,T\partial X)\otimes\mathbb{R}^3$.
Denote $\nabla$ the Levi-Civita connection of $g_{\bm\omega}$.

\begin{definition}
$\bm\omega$ is a {hypersymplectic} triple if $d\omega_i=0$. A hypersymplectic triple $\bm\omega$ is called {torsion-free} if $d(Q^{ij}\omega_j)=0$, and is called {hyperk\"ahler} if $Q_{ij}=\delta_{ij}$.
\end{definition}

Note that the torsion-free definition coincides with the one defined in the introduction by direct calculation.

Let $\nu$ be the outward unit normal vector field of $\partial X$. We are going to calculate the second fundamental form of $\partial X$, say 
$II(v,w)=\langle\nabla_v\nu, w\rangle_{\partial X}$.  Let $S\in\Gamma(\partial X,\text{End}\ T\partial X)$ be the shape operator, i.e., $\langle v,S(w)\rangle_{\partial X}=II(v,w)$. Denote $\Gamma=(\Gamma_{ij})$ the symmetric matrix $$\Gamma_{ij}=\frac{1}{2}\langle\gamma_i,d(*_{\bm\gamma} \gamma_j)\rangle_{\bm\gamma}+\frac{1}{2}\langle\gamma_j,d(*_{\bm\gamma} \gamma_i)\rangle_{\bm\gamma}=\frac{1}{2} *_{\bm\eta}(\eta_i\wedge d\eta_j+\eta_j\wedge d\eta_i)  $$ which is completely determined by $\bm\gamma$. Denote the matrix $II(e_i,e_j)$ by $A$ and $H=Tr S$.

\begin{lemma}\label{hyperkahler triple second fundamental form}
If $\bm\omega$ is a hyperk\"ahler triple, then 
$$A=\frac{1}{2}( Tr \Gamma)I-\Gamma,$$ In particular, $$2H=*_{\bm\eta}(\bm\eta\wedge d\bm\eta^T)=\langle \gamma_1, d(*_{\bm\gamma}\gamma_1)\rangle_{\bm\gamma}+\langle \gamma_2, d(*_{\bm\gamma}\gamma_2)\rangle_{\bm\gamma}+\langle \gamma_3, d(*_{\bm\gamma}\gamma_3)\rangle_{\bm\gamma},$$
$$|S|^2=Tr  {\Gamma}^2 -H^2.$$
\end{lemma}

\begin{proof}

 Fix a point $p\in \partial X$, choose a semi-geodesic coordinate system centered at $p$, say $(x^1,x^2,x^3,t)$, such that a neighborhood of $p$ is identified with $\{t\geq 0\}$
and its intersection with $\partial X$ is identified with $\{t=0\}$. The hyperk\"ahler triple can be written as 
$$\bm\omega=-dt\wedge *_{\bm\gamma_t}\bm\gamma_t+\bm\gamma_t,$$
where $\bm\gamma_t$ is a smooth family of closed framings on $\partial X$ such that $\bm\gamma_{0}=\bm\gamma$ and $$\frac{\partial \bm\gamma_t}{\partial t}=-d(*_{\bm\gamma_t}\bm\gamma_t).$$
We can further choose 
 $(x^1,x^2,x^3)$ to be a normal coordinate system for $\partial X$ at $p$ of the metric $g_{\bm\omega}|_{\partial X}$ such that $e_i=\partial_{x^i}$ at $p$. Write $\partial_{x^i}=a_i^k(x,t)e_k(x,t)$, where $\bm e_k(x,t)=(e_1(x,t),e_2(x,t), e_3(x,t))$ is the dual frame of $\bm\eta_t=*_{\bm\gamma_t}\bm\gamma_t$. Then  at $p$,
 \begin{equation}\label{II(ei,ej) calculation}
      II(e_i,e_j)=II(\partial_{x^i},\partial_{x^j})=-\frac{1}{2}\partial_t g_{ij}=-\frac{1}{2}\partial_t(a_i^ka_j^k)=-\frac{1}{2}(\partial_t a_i^k+\partial_ta_j^k).
 \end{equation}
Note that at $p$, $$\partial_t\eta_p=\partial_t a_m^pdx^m=\partial_t a_m^p\eta_m,$$ then  we have 
\begin{equation}\label{gammaidetaj calculation}
    \begin{split}
        \langle \gamma_i,d(*_{\bm\gamma}\gamma_j)\rangle_{\bm\gamma}
  &=-\langle\gamma_i,\partial_t\gamma_{j}\rangle_{\bm\gamma}\\
  &=-\langle \frac{1}{2}\delta^{ikl}\eta_k\wedge\eta_l,\frac{1}{2}\partial_t(\delta^{jpq}\eta_p\wedge\eta_q)\rangle_{\bm\gamma}\\
  &=-\frac{1}{4}\delta^{ikl}\delta^{jpq}(\partial_t a_m^p\langle \eta_k\wedge\eta_l,\eta_m\wedge\eta_q\rangle_{\bm\gamma}+\partial_ta_n^q\langle \eta_k\wedge\eta_l,\eta_p\wedge\eta_n\rangle_{\bm\gamma})\\ &=
-\frac{1}{4}\delta^{ikl}\delta^{jpq}(\partial_t a_m^p(\delta^{km}\delta^{lq}-\delta^{kq}\delta^{lm})+\partial_t a_n^q(\delta^{kp}\delta^{ln}-\delta^{kn}\delta^{lp}))\\&=-\delta^{ij}\partial_t a_k^k+\partial_t a_j^i.
    \end{split}
\end{equation}
 Combining (\ref{II(ei,ej) calculation}) and (\ref{gammaidetaj calculation}), we have 
 $$\Gamma= (Tr A) I-A,$$ so $Tr \Gamma=2Tr A$, $A=\frac{1}{2}(Tr \Gamma) I-\Gamma$.

\end{proof}

\section{General Riemannian geometry}\label{section Riemannian geometry}

\subsection{Evolution equations of hypersurfaces}
We refer to \cite{petersen2016riemannian} Section 3.2 and \cite{koiso1981hypersurfaces} for discussions in this section.

Let $M$ be a Riemannian manifold, $f$ be a smooth distance function on $M$, i.e., $|\nabla f|= 1$, so $\nabla_{\nabla f}\nabla f=0$. The $(1,1)$ tensor  corresponds to  $\text{Hess} f$ is
\begin{equation}\label{(1,1) tensor shape operator}
    S(X)=\nabla_X\nabla f,
\end{equation}
 and its trace is $H=\Delta f\in C^\infty(M)$. 
  When $\Sigma\subset M$ is a hypersurface defined by a level set of $f$, the restriction of $S$ to $T\Sigma$ has image in $T\Sigma$, which is the shape operator of $\Sigma$. The second fundamental form of $\Sigma$ with respect to $\nabla f$ is  $II(X,Y):=\langle S(X), Y\rangle=\text{Hess} f(X,Y)$, so $H$ is the mean curvature and $\overrightarrow{H}=-H\nabla f$ is the mean curvature vector.
  By tensor calculations, we have evolution equations of second fundamental forms

\begin{equation}\label{simplified radial curvature equation}
L_{\nabla f} S+S^2=-R(\cdot,\nabla f)\nabla f,
\end{equation}
\begin{equation}\label{radial curvature equations''}
L_{\nabla f} \text{Hess} f-\text{Hess}^2 f=-Rm(\cdot,\nabla f,\cdot,\nabla f),
\end{equation}
where $$\text{Hess}^2 f(X,Y)=\langle S^2(X),Y \rangle=\langle S(X),S(Y)    \rangle,$$
and one has the equality
\begin{equation}\label{Lie derivative S same as covariant derivative S}
    L_{\nabla f}S=\nabla_{\nabla f} S.
\end{equation}
Take the trace of $(\ref{simplified radial curvature equation})$, we have
\begin{equation}\label{evolution of mean curvature}
L_{\nabla f} H=-|S|^2-\text{Ric}(\nabla f,\nabla f). 
\end{equation}
where $|S|^2:=Tr(S^2)$ is the norm square of the shape operator.

Besides the evolution equations,  we have Gauss equations on $\Sigma$ 
\begin{equation}
    \begin{split}
        Rm_{M}(X,W,Y,Z)=&Rm_{\Sigma}(X,W,Y,Z)+II(X, Z)II(W, Y)\\&-II(X, Y)II(W, Z).
    \end{split}
\end{equation} 
Take the trace with respect to $W,Z$, we have 
\begin{equation}\label{trace gauss equation}
\text{Ric}_{M}=\text{Ric}_{\Sigma}+\text{Hess}^2 f-H\cdot \text{Hess}f+Rm_{M}(\cdot, \nabla f,\cdot,\nabla f),
\end{equation}
take the trace again, we have
\begin{equation}\label{scalar curvatures and hypersurface}
    R_M=R_{\Sigma}+|S|^2- H^2+2\text{Ric}_M(\nabla f,\nabla f),
\end{equation}
 where $R$ denote scalar curvatures.
Use equations (\ref{radial curvature equations''}) and (\ref{trace gauss equation}) to cancel the curvature term involving $\nabla f$, we get 
\begin{equation}\label{evolution equation for 2nd form}
L_{\nabla f} \text{Hess} f=\text{Ric}_{\Sigma}-\text{Ric}_M+ 2\text{Hess}^2 f-H \cdot \text{Hess}f.
\end{equation}

\subsection{The boundary exponential map}

Let $(M,g)$ be a complete Riemannian manifold with boundary, which means the induced metric space is complete. 
Denote $T^\perp \partial M$  the normal line bundle of $\partial M$, which is a trivialized by the inward unit normal vector field $N$. We identify $T^\perp\partial M$ with $\partial M\times \mathbb R$ via this trivialization. For $p\in\partial M$, denote $\gamma_p(t)$ the geodesic such that $\gamma_p(0)=p,\gamma_p'(0)=N_p$.  Denote $$D(p)=\inf\{t>0| \gamma_p(t)\in\partial M\}\in (0,\infty],$$ $$\tau(p)=\sup\{t>0|d(\gamma_p(t),\partial M)=t\}\in(0,\infty].$$
We have a subset of $U_{\partial M}\subset T^\perp\partial M$ defined by
 $$U_{\partial M}=\{(p,tN_p)\in T^{\perp}\partial M|0\leq t <D(p) \},$$ which is the domain of the \text{boundary exponential map}
 \begin{equation}\label{boundary exponential map definition}
      \exp^\perp:U_{\partial M}\rightarrow M, (p,s)\mapsto\gamma_p(s).
 \end{equation}
 and define $$V_{\partial M}=\{(p,tN_p)\in T^{\perp}\partial M|0\leq t <\tau(p) \}\subset U_{\partial M}.$$

 There are some definitions, notations and terminologies related to the boundary exponential map that appear in this paper:
 \begin{itemize}
     \item 
 The \emph{boundary injectivity radius} $i_b$ is defined to be the supremum of $s\geq 0$ such that $\exp^\perp|_{\partial M\times[0,s)}$ is a diffeomorphism onto its image.
 \item A \emph{focal point} $q$ of $\partial M$ is a critical value of the boundary exponential map (\ref{boundary exponential map definition}). If $q$ lies in $\gamma_p$ for some $p\in\partial M$, we say $q$ is a focal point along $\gamma_p$. 
 \item A \emph{foot point} of $q\in M$ is a point  $p\in \partial M$ such that $d(q,p)=d(q,\partial M).$ 
 \item A \emph{cut point} of $\partial M$ is a point $q\in M$ such that there exists a foot point $p$ of $q$ such that $d(q,p)=\tau(p)$. We also say $q$ is a cut point of $p$.
 \item When we say a covariant tensor on $M$ is written in \emph{geodesic gauge}, we mean the pull back of this tensor via $\exp^\perp$.
 \item 
 For a  subset $B\subset\partial M$, we use the notation  $$C(B,t_1,t_2)=\exp^\perp(B\times[t_1,t_2))$$ to denote a metric cylinder with base $B$.
 \item $N_r(\partial M,g
):=\{x\in M|d(x,\partial M)\leq r\}.$
 \item The $(1,1)$ tensor $S$ in (\ref{(1,1) tensor shape operator})  is defined with respect to $f=-d(\cdot,\partial M)$ near $\partial M$, so $\nabla f=-N$ on $\partial M$, and $II\geq 0$ if $\partial M$ is convex.
 \end{itemize}
 
Here are some remarks about some of these definitions:
 
 \begin{itemize}
     \item By definition,  $\gamma_p(s)$ is a focal point of along $\gamma_p$, if and only if there exists a non-zero \emph{$\partial M$-Jacobi field} $V$ along $\gamma_p$ (a Jacobi field with $ V(0)\in T_p \partial M, V'(0)+S(V(0))\in T_p^\perp \partial M$) such that $V(s)=0$. If there is no focal point along $\gamma_p|_{[0,l)}$, then 
$$I_0(W,W)=\int_{0}^l \langle W', W'\rangle-\langle R(W,\gamma_p')\gamma_p',W \rangle dt- \langle S(W),W\rangle(0)\geq 0 $$ for any piecewise smooth vector field $W$ along $\gamma$ with $W(0)\in T_{\gamma(0)}\partial M$. 
\item 
By Lemma 3.2 in \cite{sakurai2017rigidity}, $\tau$ defines a continuous map from $\partial M$ to $(0,\infty]$.  It is well-known that by a second variation argument, the first focal point along $\gamma_p$ appears no later than $\tau(p)$, and moreover one can argue by contradiction to get (See Lemma 3.6 in \cite{sakurai2017rigidity}) 
 \end{itemize}

\begin{proposition}\label{first focal point two foot point}
$q\in M$ is a cut point of $p$ if and only if at least one of the following holds:
\begin{itemize}
    \item $q$ is the first focal point of $\gamma_p$ ;
    \item $q$ has at least two foot points.
\end{itemize}
\end{proposition}
\noindent From this,  we have $\exp^\perp|_{V_{\partial M}}$ is a diffeomorphism and
$$i_b=\inf_{p\in\partial M}\tau (p).$$

It is worth noting that if $M$ is embedded in some complete Riemannian manifold $M'$ of the same dimension and view $\partial M$ as an embedded hypersurface of $M'$, then it may happen that a ``focal point'' of $\partial M$ in $M'$ lies outside $M$ if we define the boundary exponential map in the whole normal bundle $T^\perp\partial M$. However, our definition is intrinsic for $M$.

\subsection{Manifolds with mean convex boundary}

Now we focus on manifolds with mean convex boundary. We will summarize some results and discuss a new result.

The following results are well-known,  see for example \cite{li2014sharp} and \cite{Donaldson2017BoundaryVP}.

\begin{proposition}\label{diam}
Let $(M,g)$ be a compact, connected Riemannian manifold with boundary, $\emph{Ric}_M\geq 0$. Suppose $\partial M$ has mean curvature $H\geq H_0>0$, then $\partial M$ is connected,

$$\pi_1(M,\partial M)=0,$$ $$\sup_{q\in M}d(q,\partial M)\leq (n-1)H_0^{-1},$$  $$\emph{vol}(M)\leq C(n)H_0^{-1}\emph{vol}(\partial M).$$
\end{proposition}

\begin{proof}

If $\pi_1(M,\partial M)\neq 0$ or $\pi_0(\partial M)\neq 0$, then every non-trivial class contains a non-trivial unit speed geodesic $\gamma:[0,l]\rightarrow M $ which minimize the length of all curves in its class.
From the first variation formula, $\gamma$ intersects boundary perpendicularly at both end points. Pick an orthonormal basis $V_i,1\leq i\leq n-1$ of $T_{\gamma(0)}\partial M$ and parallel them transport along $\gamma$ to get $V_i(t)$. Let $\gamma_{i,s}(t)$ be a family of curves centered at $\gamma $ with variation field $V_i(t)$. By second variational formula, $$0\leq \sum_{i=1}^{n-1} \frac{d^2}{ds^2}\Big{|}_{s=0}E(\gamma_{i,s})=\int_0^l -\text{Ric}(V_i(t),V_i(t))dt-H(\gamma(0))-H(\gamma(l))<0,$$ which is a contradiction.

If for some $q\in M$, $\tilde{l}:=d(q,\partial M)> (n-1)H_0^{-1}$. Let $p$ be a foot point of $q$. Let $V_i,1\leq i\leq n-1$ be an orthonormal basis of $T_{p}\partial M$ and parallel them transport along $\gamma_p$ to get $V_i(t)$, and denote $\tilde{V}_i(t)=(\tilde l-t)V_i(t)$, $\tilde \gamma_{i,s}(t)$  a family of curves centered at $\gamma_p $ with variation field $\tilde{V}_i(t)$, then $$0\leq \sum\limits_{i=1}^{n-1}\frac{d^2}{ds^2}\Big{|}_{s=0}E(\tilde\gamma_{i,s})=(n-1)\tilde l-\int_0^{\tilde{l}} \text{Ric}(tV_i(t),tV_i(t))dt-\tilde{l}^2H(\gamma(0))<0,$$ which is a contradiction.
The volume upper bound is by volume comparison, see \cite{heintze1978general}.
\end{proof}

\begin{remark}
Note that when $M$ is connected,   $\pi_1({M},\partial M)=0$ is equivalent to that $\partial M$ is connected and the natural map $\pi_1(\partial M)\rightarrow \pi_1(M)$ is surjective, the latter of which means that fix a point $p_0\in\partial M$, then
for any closed path $x(t):0\leq t\leq 1$  in $M$ with $x(0)=x(1)=p_0$, there exists a homotopy $F_s(t):0\leq s, t\leq 1$ with $F_s(0)=F_s(1)=p_0, F_0(t)=x(t)$ such that $F_1(t)\in \partial M$.
\end{remark}

The following result is well-known, see Lemma 6.3 of \cite{kodani1990convergence}.

\begin{proposition}
Let $M$ be a complete Riemannian manifold with  nonempty  compact boundary. If there are no focal point whose distance to $\partial M$ is equal to $i_b$, then there exists a smooth geodesic of length $2i_b$ which is perpendicular to $\partial M$ at both end points.
\end{proposition}

\begin{proof}
$i_b=\inf\limits_{p\in\partial M}\tau (p)>0$. Suppose the infimum is achieved at $p_1\in\partial M$. Then by assumption and Proposition \ref{first focal point two foot point} $\gamma_{p_1}(i_b)$ has another foot point $p_2$. We claim $\gamma_{p_1}'(i_b)=-\gamma_{p_2}'(i_b)$, so $D(p_1)=D(p_2)=2i_b$ and $\gamma_{p_1}:[0,2i_b]\rightarrow M$ is the smooth geodesic we want. By assumption, we can find smooth distance functions $h_1, h_2$  extending $d(\cdot,\partial M)$  near $\gamma_{p_1}|_{[0,i_b]}$, $\gamma_{p_2}|_{[0,i_b]}$, respectively. Consider the smooth hypersurface $\Sigma=(h_1-h_2)^{-1}(0)$ near $q:=\gamma_{p_1}(i_b)=\gamma_{p_2}(i_b)$. Then $v=\nabla h_1(q)+\nabla h_2(q)\in T_{q}\Sigma$. If it is non-zero, then $\langle \nabla h_1(q)+\nabla h_2(q) , v \rangle>0$. Without loss of generality, assume $\langle\nabla h_1,v\rangle>0$. Then in the direction of $-v$ in $\Sigma$, we have some point $q'\in\Sigma$ with $h_1(q')<h_1(q)$. Hence $q'$ has two foot points and $d(q',\partial M)=h_1(q')<i_b$, which is a contradiction.
\end{proof}
It is worth noting that in this proof, Kodani used the first order variation of $h_1$ on $\Sigma$ to lead a contradiction. We can also investigate the second order variation of $h_1$ on $\Sigma$ and prove the following:

\begin{proposition}\label{existence of focal points}
Let $M$ be a compact Riemannian manifold with mean convex boundary, $\emph{Ric}_M\geq 0$,  then there exists a focal point of $\partial M$ whose distance to $\partial M$ is equal to $i_b$.
\end{proposition}

\begin{proof}
Suppose not, by the previous proposition, we have a smooth geodesic of length $2i_b$ which is perpendicular to $\partial M$ at both end points. We use notations $p_1,p_2, q, h_1,h_2,\Sigma$ as in the previous proposition. We claim $\Delta_\Sigma h_1(q)<0$, so we get another point $q''\in \Sigma$ near $p$ with $h_1(q'')<h_1(q)=i_b$ and get a contradiction. Denote $N_{0}=\nabla h_1(q)=-\nabla h_2(q),\Sigma_1=h_1^{-1}(i_b),\Sigma_2=h_2^{-1}(i_b)$, then $N_0$ is a common unit normal vector for $\Sigma,\Sigma_1,\Sigma_2$ at $q$. Let $II_\Sigma, II_{\Sigma_1},II_{\Sigma_2}$ be second fundamental forms with respect to $N_0$ at $q$.
Then at $q$, $$II_\Sigma= \frac{1}{|\nabla (h_1-h_2)|}\text{Hess} (h_1-h_2)=\frac{1}{2} \text{Hess}(h_1-h_2)=\frac{1}{2}(II_{\Sigma_1}+II_{\Sigma_2}),$$  hence 
 $$H_\Sigma=\frac{1}{2}(H_{\Sigma_1}+H_{\Sigma_2}),
\overrightarrow{H}_{\Sigma}=\frac{1}{2}(\overrightarrow{H}_{\Sigma_1}+\overrightarrow{H}_{\Sigma_2}).$$From the formula of Laplace operator on a hypersurface, we know that at $q$,
$$\Delta_\Sigma h_1=\Delta h_1-\text{Hess} h_1(N_0,N_0)+\langle \nabla h_1, \overrightarrow{H}_\Sigma \rangle ,$$
$$\Delta_{\Sigma_1} h_1=\Delta h_1-\text{Hess} h_1(N_0,N_0)+\langle \nabla h_1,\overrightarrow{H}_{\Sigma_1} \rangle.$$
Since $h_1$ is a constant on $\Sigma_1$, $\Delta_{\Sigma_1}h_1=0$, hence
\begin{equation}
\begin{split}
\Delta_{\Sigma} h_1 &=\Delta_{\Sigma} h_1-\Delta_{\Sigma_1}h_1=\langle\nabla h_1, \overrightarrow H_{\Sigma}-\overrightarrow H_{\Sigma_1}\rangle=\langle \nabla h_1, \frac{1}{2}(\overrightarrow H_{\Sigma_2}-\overrightarrow H_{\Sigma_1})\rangle \\
&=-\frac{1}{2}(H_{\Sigma_2}-H_{\Sigma_1}).
\end{split}
\end{equation}
Since $\partial M$ is mean convex, $\text{Ric}_M\geq 0$, by the evolution equation of mean curvature (\ref{evolution of mean curvature}), we have $-H_{\Sigma_1}(q)> H_{\partial M}(p_1)>0$, $H_{\Sigma_2}(q)> H_{\partial M}(p_2)>0$. Hence $\Delta_{\Sigma} h_1(p)<0$, which completes the proof.

\end{proof}

A moment thought about the arguments in the end of the previous proof yields that $\text{Ric}_M\geq 0$ is not so necessary, since we can make use of the evolution equation (\ref{evolution of mean curvature}) to get an ordinary differential inequality for the mean curvature. 

\begin{proposition}\label{more general inj lower bound}
If in the previous proposition we assume instead $\emph{Ric}_M\geq -(n-1)c$ for some $c>0$, and $H\geq  H_0>0$. If $i_b< -\frac{1}{2}(n-1)\ln \big{|}\frac{H_0-(n-1)\sqrt{c}}{H_0+(n-1)\sqrt{c}}\big{|}$, then 
there exists a focal point of $\partial M$ whose distance to $\partial M$ is equal to $i_b$. 
\end{proposition}

\begin{proof} Suppose the conclusion is not true, follow the arguments as before except for second last sentence. Let  $S_i(X)=-\nabla_{X}\nabla h_i$, $H_i=Tr S_i=-\Delta h_i$, and identify a neighborhood of $\gamma_{p_i}|_{[0,i_b]}$ with a subset of $\partial M\times\mathbb{R}$ via $\exp^\perp$. Then by (\ref{evolution of mean curvature}),$$\partial_t{ H_i}=|S_i|^2+\text{Ric}(\nabla h_i,\nabla h_i)\geq\frac{1}{n-1}H_i^2-(n-1)c$$ and $H(p_i,0)\geq H_0$.
Let $f$ solves the ODE on $[0,i_b]$
 $$ f'=\frac{1}{n-1}f^2-(n-1)c$$ and $f(0)=H_0$, then we have $H_i(p_i,t)\geq f(t)  $.  In particular, $H_i(p_i,i_b)\geq f(i_b)>0$, which leads to a contradiction as before.

\end{proof}

Proposition $\ref{existence of focal points}$ implies

\begin{corollary}\label{Ric positive, mean positive, inj lower bound}
Let $(M,g)$ be a compact Riemannian manifold with boundary, $K>0,\lambda>0$ are constants.
Suppose $\sec \leq K$, $S\leq \lambda$, $H>0$, $\emph{Ric}_M \geq 0$, then $i_b\geq \frac{1}{\sqrt{K}}\emph{arccot}{\frac{\lambda}{\sqrt{K}}}.  $
\end{corollary}

\begin{proof}
By Proposition \ref{existence of focal points}, there exists $p\in\partial M$ such that $\gamma_p(i_b)$ is a focal point along $\gamma_p$. If $i_b<\frac{1}{\sqrt{K}}\text{arccot}{\frac{\lambda}{\sqrt{K}}} $, from comparison theorem for Jacobi fields, we know  $\gamma_p(i_b)$ cannot be a focal point along $\gamma_p$, which is a contradiction. 
\end{proof}

Similarly, Proposition \ref{more general inj lower bound} implies

\begin{corollary}\label{inj lower bound limit}
Let $(M,g)$ be a compact Riemannian manifold with boundary. Suppose $|Rm|\leq C$, $|S|\leq C$, $H\geq H_0>0$, then we can find $i_0$ depending explicitly on $C,H_0$ such that $i_b\geq i_0.$
\end{corollary}

\begin{remark}\label{inj lower bound limit remark}
In the previous two corollaries, if the sectional curvature and Ricci curvature bounds only holds for $N_1(\partial M,g)$, then we also have a $i_b$ lower bound. In the case of Corollary \ref{Ric positive, mean positive, inj lower bound}, we have $i_b\geq \min\{\frac{1}{\sqrt{K}}\text{arccot}{\frac{\lambda}{\sqrt{K}}} ,1\}$. 
\end{remark}

\begin{remark}
In the same setting as the previous two corollaries, \cite{knox2012compactness} Lemma 2.2 claimed to prove a lower bound for $i_b$, using a similar method as \cite{Anderson2012BoundaryVP} Lemma 2.4.  In both papers, there is a logic problem that they get a contradiction with an unjustified statement: Let $M$ be a Riemannian manifold with boundary, $\gamma:[0,l]\rightarrow M$ be a geodesic that is perpendicular to the boundary at both end points, and suppose there is no focal point along $\gamma$ for both boundary portions, then $I_1(V,V)\geq 0$ for any smooth vector field along $\gamma$ with $V(0),V(l)\in T{\partial M}$. Here $$I_1(V,V)=\int_{0}^l \langle V', V'\rangle-\langle R(V,\gamma')\gamma',V \rangle dt- \langle S(V(0)),V(0)\rangle-\langle S(V(l)),V(l)\rangle. $$ In fact, this unjustified statement is not true, and one can easily think of an example: let $$\Sigma_1=\{(x',x^n)\in\mathbb{R}^n\big | |x'|^2+(1-x^n)^2=R_1^2\},$$ $$\Sigma_2=\{(x',x^n)\in\mathbb{R}^n\big | |x'|^2+(1+x^n)^2=R_2^2\},$$ $R_1,R_2>2$
and $\gamma(t)=(0,1-t), 0\leq t\leq 2 $, then $I_1(V,V)=-\frac{1}{R_1}-\frac{1}{R_2
}<0$ for any unit-norm parallel vector field along $\gamma$ with $V(0)\in T_{\gamma(0)}\Sigma_1$. In this case, there exist no focal points on $\gamma$ for both $\Sigma_1,\Sigma_2$.

In fact, focal points give crucial information for index form defined by one submanifold and one point. However, as seen from the example, they do not fit well with the index form defined for two submanifolds. Indeed, there is some notion of ``conjugate point'' defined for two submanifolds, see \cite{ambrose1961index}. 
\end{remark}

It is easy to see focal points can ``pass to the limit'', since they arise from kernels of the differential of exponential maps. Though one can use Corollary \ref{inj lower bound limit} directly in many situations, we point out this fact here, which may help in contradiction arguments.

\begin{proposition}\label{pass to limit}
Let $M$ be a manifold with boundary, $g_i$ be a sequence of Riemannian metrics on $M$, and $p_i\in\partial M$. Suppose $g_i$ converges to a Riemannian metric $g_{\infty}$ smoothly and $p_i$ converges to $p_\infty\in\partial M$, $\gamma_{p_i}$ is defined on $[0,b]$ and
$\gamma_{p_i}(t_i)$ is a focal point along $\gamma_{p_i}$ with $0<a\leq t_i\leq b$. Then for a subsequence, $\gamma_{p_i}$ converges smoothly to $\gamma_{p_\infty}$, $t_i\rightarrow t_\infty$ and $\gamma_{p_\infty}(t_\infty)$ is a focal point along $\gamma_{p_\infty}$.
\end{proposition}

\begin{proof}
$\gamma_{p_i}'(0)\in TM$ is a bounded sequence, hence subconverges to some $v\in TM $, which must equal to $\gamma_{p_\infty}'(0)$. Hence by ODE theories, $\gamma_{p_i}$ converges smoothly to $\gamma_{p_\infty}$. Suppose for a subsequence $t_i\rightarrow t_\infty$. Let $J_i:[0,t_i]\rightarrow TM$ be $\partial M$-Jacobi fields along $\gamma_i$ with $J_i(t_i)=0, |J_i'(t_i)|=\frac{t_\infty}{t_i}$. Normalize these geodesics and Jacobi-fields by $\bar\gamma_i(t)={\gamma}_i(\frac{t_i}{t_\infty} t)$, $\bar{J}_i(t)=J_i(\frac{t_i}{t_\infty} t)$, so  $\bar\gamma_i, \bar{J}_i$ are defined on the same interval $[0,t_\infty]$, and
\begin{equation}
\begin{split}
\bar J_i''(t)&+  Rm_{g_i}(\bar{J}_i(t),\bar\gamma_i'(t))\bar\gamma_i'(t)=0,\\
 \bar{J}_i(t_\infty)&=0,|\bar{J}_i'(t_\infty)|=1.
\end{split}\end{equation}
For a subsequence $\bar{J_i}'(t_\infty)\rightarrow w$ with $|w|=1$. Let ${J_\infty}$ be the non-trivial Jacobi-field along $\gamma_{p_\infty}$ with $J_\infty(t_\infty)=0, J'_\infty(t_\infty)=w$, then $J_i$ converges smoothly as maps $[0,t_\infty]\rightarrow TM$ to $J_\infty$ by ODE theories. Hence $J_\infty$ is a $\partial M$-Jacobi field, which implies $\gamma_{p_\infty}(t_\infty)$ is a focal point along $\gamma_{p_\infty}.$
\end{proof}

\subsection{Volume estimates near boundary}

In this subsection, we show some volume lower bounds near boundary under some geometric control. These estimates can be found in \cite{anderson2004boundary} and \cite{knox2012compactness}.

\begin{proposition}\label{volume of a metric cyclinder}
Let $({M},g)$ is a Riemannian manifold with boundary, suppose $\emph{Ric}\geq -(n-1)c$ for some $c\geq 0$, and $\exp^\perp$ is an diffeomorphism in $B\times [0,T)$ for some open subset $B$ of $\partial M$, then $$\sup_{B\times \{t\}} H\leq\max\{(n-1)\sqrt{2c},4(n-1)^{-1}T^{-1}\},$$
$$\emph{vol}(C(B,t_1,t_2))\geq C(n,T,c)\emph{vol}_{\partial M}(B)|t_2-t_1|,$$ where $0\leq t, t_1,t_2\leq\frac{1}{2}T.$

\end{proposition}

\begin{proof}
By the evolution equation (\ref{evolution of mean curvature}),
\begin{equation}
\partial_t H=|S|^2+\text{Ric}(N_t,N_t),
\end{equation}
hence
\begin{equation}
\partial_t H\geq \frac{1}{n-1}H^2-(n-1)c.
\end{equation}
If for some $t_0\in[0,\frac{1}{2}T],z_0\in B$, we have $H(z_0,t_0)\geq \delta$ and $\delta\geq (n-1)\sqrt{2c}$, then $\partial_t H(z_0,t)\geq 0$ and $H(z_0,t)\geq (n-1)\sqrt{2c}$ for $t\in[t_0,T) $. Hence $$\partial_t H(z_0,t)\geq \frac{1}{2}(n-1)^{-1}H(z_0,t)^2,$$ $$H^{-1}(z_0,t)\leq H^{-1}(z_0,t_0)-\frac{1}{2}(n-1)(t-t_0)\leq \delta^{-1}-\frac{1}{2}(n-1)(t-\frac{1}{2}T).$$ Then the continuity of $H(z_0,\cdot)$  in  $[0,T)$ forces $\delta\leq 4(n-1)^{-1}T^{-1}$, which proves the mean curvature estimate.

For $t\in[0,\frac{1}{2}T]$,  let $B_t=\exp^\perp(B\times \{t\})$, then $$\frac{d}{dt}\mathcal{H}^{n-1}(B_t)=-\int_{B_t} Hd\mathcal{H}^{n-1}(B_t)\geq-C_1(n,T,c)\mathcal{H}^{n-1}(B_t) ,$$where $\mathcal{H}^{n-1}$ denotes the $(n-1)$-dimensional Hausdorff measure of ${M}$. Hence $$\mathcal{H}^{n-1}(B_t)\geq e^{-C_1(n,T,c)t}\text{vol}_{\partial M}(B)\geq e^{-\frac{1}{2}C_1(n,T,c)T}\text{vol}_{\partial M}(B)$$ and when $0\leq t_1<t_2\leq \frac{1}{2}T$, $$\text{vol}(C(B,t_1,t_2))=\int_{t_1}^{t_2}\mathcal{H}^{n-1}(B_t)dt\geq C(n,T,c)\text{vol}_{\partial M}(B)(t_2-t_1).$$
\end{proof}

\begin{proposition}\label{volume noncollasping}
Let $M$ be a Riemannian manifold with boundary, $p\in\partial M$. Suppose  $B(p,2r_0)$ has compact closure,
$$\sup\limits_{B(p,2r_0)}|Rm|\leq C,$$
$$\emph{vol}_{\partial M}(B_{\partial M}(p,r_0))\geq v_0,$$
  $\exp^\perp$ is an diffeomorphism in $B_{\partial M}(p,r_0)\times [0,r_0)$, and on $B_{\partial M}(p,r_0)$ $$  \emph{Ric}_{\partial M}\geq -(n-2)c_0,  \ |S|\leq C,$$  then there exists $r_1>0,v_1>0$ depending on $n, C,c_0, r_0,v_0$, such that for $q=\exp^\perp(p,2r_1)$, we have $$\emph{vol}(B(q,r_1))\geq v_1.$$

\end{proposition}

\begin{proof}

 By definition $C(B_{\partial M}(p,r_0),0,r_0)\subset B(p,2r_0)$.  We claim that  there exists $r_2>0,C_1>0$ such that for any $p_1,p_2\in B_{\partial M}(p,r_0)$, $$d_{\Sigma_t}(\exp^\perp(p_1,t),\exp^\perp(p_2,t))\leq C_1d_{\partial M}(p_1,p_2)$$ when $0\leq t\leq r_2$, where $\Sigma_t$ is the image of $B_{\partial M}(p,2r_0)$ under $\exp^\perp(\cdot, t)$. In fact, by (\ref{simplified radial curvature equation})(\ref{Lie derivative S same as covariant derivative S}),
 $$\nabla_{\nabla t } S=S^2+R(\cdot,\nabla t)\nabla t, $$ hence  $$\frac{d}{dt}|S|\leq |S|^2+C.$$
Integrate the inequality, we have 
$$\arctan(\frac{|S|}{\sqrt C})(x,t)-\arctan(\frac{|S|}{\sqrt C})(x,0)\leq \sqrt{C}t.$$
Hence there exist $r_2>0,C_2>1$ such that $|S|\leq \log C_2$ when $0\leq t\leq r_2$.
Now let $\gamma_0(s)$ be a smooth curve in $B_{\partial M}(p,2r_0)$ that connects $p_1,p_2$, and let $\gamma_t(s)=\exp^\perp(\gamma_0(s),t)\in \Sigma_t$, then we have $$\frac{d}{dt} \log |\gamma_t'(s)|=-\frac{\langle S_{\Sigma_t}(\gamma_t'(s)),\gamma_t'(s)\rangle}{\langle\gamma_t'(s),\gamma_t'(s)\rangle}\leq |S_{\Sigma_t}(\gamma_t(s))|\leq \log C_2.$$
It follows that $|\gamma_t'(s)|\leq C_1|\gamma_0'(s)|$ with $C_1=C_2^{r_2}$ and the claim follows from integration. 
Now take $$r_1=\min\{2C_1r_0,\frac{1}{4}r_2\},$$ we have $$C(B_{\partial M}(p,\frac{r_1}{2C_1}),\frac{3r_1}{2},\frac{5r_1}{2})\subset B(q,r_1),$$ then we apply Proposition \ref{volume of a metric cyclinder} and Bishop-Gromov volume comparison on $\partial M$ to get the desired conclusion.
\end{proof}

\begin{remark}
It is easy to give a quantitative version of the lemma from the proof. However, to the author's knowledge, we cannot prove the last inclusion in the proof without a control of curvature. It may be possible that a metric ball of the boundary becomes ``long and thin'' under the flow of $\nabla d(\cdot,\partial M)$,  while maintains an area lower bound. 
\end{remark}

\subsection{Harmonic radius and convergence theory}

Convergence theory of Riemannian manifolds is a powerful tool to prove conclusions in Riemannian geometry through contradiction arguments when explicit bounds is not required. In this section, we will restate some results of \cite{anderson2004boundary}, follow the proof there, and discuss some direct corollaries.

Let $(M,g)$ be a Riemannian manifold with boundary, $m\in \mathbb{N}$,   $0<\alpha<1$, $Q>1$. For $p\in M$, define $r_h^{m,\alpha}(p,g,Q)$ to be the supremum of $\rho>0$ such that if $d(p,\partial M)>\rho$,
then there exists a neighborhood $U$ of $p$ in $M$ and a interior coordinate chart $\varphi:B_{\frac{\rho}{2}}\rightarrow U$, $\varphi(0)=p$, and  if $d(p,\partial M)\leq\rho$, then there exists a neighborhood $U$ of $p$ in $M$ and a boundary coordinate chart
$\varphi: B_{4\rho}^+\rightarrow U$, $\varphi((0,d(p,\partial M)))=p, \varphi(\tilde{B}_{4\rho})=U\cap\partial M$, and in either $B_{\frac{\rho}{2}}$ or $B_{4\rho}^+$, we have $$\Delta_M\varphi^{-1}=0,$$
$$Q^{-2}(\delta_{ij})\leq (g_{ij})\leq Q^2 (\delta_{ij}), $$
$$\rho^{m+\alpha}\sum_{|\beta|=m}|\partial_\beta g_{ij}(x)-\partial_\beta g_{ij}(y)|\leq (Q-1)|x-y|^\alpha$$
We call such a coordinate chart a $(\rho,Q,m,\alpha)$-harmonic coordinate chart centered at $p$. Note that the second condition implies there exists $r_1,r_2$, depending on $\rho,Q$, such that
$B(p,r_1)\subset U\subset B(p,r_2)$.

\begin{definition}
Fix an integer $m\geq 0$, and $0<\alpha<1$.  We say a sequence of Riemannian manifold with boundary $(M_i,g_i,p_i)$ converges in pointed $C^{m,\alpha}$ to $(M_\infty,g_\infty,p_\infty)$ if there exists precompact open subsets  $\Omega_{i}$ of $M_i$ and $\Omega_{\infty,i}$ of $M_\infty$, and $\sigma_i>\rho_i\rightarrow\infty$ such that $B(p_i,\rho_i)\subset \bar{\Omega}_{i}\subset B(p_i,\sigma_i)$, $B(p_\infty,\rho_i)\subset \bar{\Omega}_{i}\subset B(p_\infty,\sigma_i)$ and  there exists diffeomorphisms $F_{i}:\Omega_{\infty,i}\rightarrow \Omega_{i}$, $F_{i}:\Omega_{\infty,i}\cap\partial M_i\rightarrow\Omega_{i}\cap\partial M_\infty$ such that  $F_{i}^*g_i\rightarrow g$ in $C^{m,\alpha}$ topology,  and $F_{i}^{-1}(p_i)\rightarrow p_\infty$. If we replace $C^{m,\alpha}$ by $C^\infty$, we say the convergence is in pointed Cheeger-Gromov sense.
\end{definition}

\begin{remark}
$(M_\infty, g_\infty)$ is automatically a complete $C^{m,\alpha}$ or $C^\infty$ Riemannian manifold with boundary from the definitions. Sometimes we only need that one metric ball converges, so one can modify the definitions above: suppose $\bar{B}(p_i,r)\subset \Omega_i$ for some precompact open set $\Omega_i\subset M_i$ and there exists a Riemannian manifold with boundary $(\Omega_\infty,g_\infty)$, a point $p_\infty\in\Omega_\infty$, and  diffeomorphisms $F_i:\Omega_{\infty}\rightarrow\Omega_i$ mapping $\partial \Omega_\infty$ onto $\Omega_i\cap\partial M_i$ such that $F_i^*g_i\rightarrow g_\infty$ in $C^{m,\alpha} $ or $C^\infty$ topology and $F_i^{-1}(p_\infty)\rightarrow p_i$, we say $B(p_i,r)$ converges in $C^{m,\alpha}$  or Cheeger-Gromov sense to $B(p_\infty, r)$.

\end{remark}

The following theorem is well-known and is a fundamental theorem of Riemannian convergence theory. 

\begin{proposition}\label{harmonic radius lower bound}
Let $(M_i,g_i)$ be  a sequence of  complete Riemannian manifold with boundary, $p_i\in M_i$. Suppose there exists some $Q>1$, and a positive function $r:(0,\infty)\rightarrow(0,\infty) $, such that $r_h^{m,\alpha}(p,g_i,Q)\geq r(R)$ for any $p\in B(p_i,R)$, then for a subsequence, $(M_i,g_i,p_i)$ converges in pointed $C^{m,\beta}$ sense to $(M_\infty,g_\infty,p_\infty)$ for any $0<\beta<\alpha$. If the above assumption holds for only one $R$, then $B(p_i,R)$ converges in $C^{m,\beta}$ sense to $B(p_\infty, R)$.

\end{proposition}

Next, we discuss under what geometric control we can  get a harmonic radius lower bound. We state and prove the following local version of Theorem 3.2.1 in  \cite{anderson2004boundary}, with simplified arguments in some parts.

\begin{theorem}\label{local harmonic radius lower bound}
Fix $m\geq 1$. Let $(M,g)$ be a  Riemannian manifold with boundary, and $\Sigma\subset\partial M$ be a boundary metric ball with compact closure, nonempty boundary. Suppose $\exp^\perp$ maps $\Sigma\times [0,i_0)$ diffeomorphically onto its image $\Omega$,  
\begin{equation}\label{injectivity radius explaination}
     inj_{\Omega}\geq i_0,\ inj_{\Sigma}\geq i_0,
\end{equation}
in $\Omega$,
\begin{equation}\label{bounds 1 up to m}
    |\nabla^l \text{Ric}_M|\leq \Lambda, 0\leq l\leq m,
\end{equation}
 on $\Sigma$
 \begin{equation}\label{bounds 2 up to m}
     |\nabla_{\partial M}^l \text{Ric}_{\partial M}|\leq \Lambda,   |\nabla_{\partial M}^{l+1} H|\leq \Lambda, 0\leq l\leq m.
 \end{equation}
 Then for any $Q>1$, $\alpha\in(0,1)$, $p\in\Omega$, 
 \begin{equation}\label{harmonic radius lower bound non empty boundary}
     r^{m+1,\alpha}_h(p,g,Q)\geq r_0(i_0,\Lambda,m,\alpha,Q)d(p,\partial ^+\Omega),
 \end{equation}
  where $\partial^+\Omega=\bar {\Omega} \backslash(\Omega\cup\partial M)$.
\end{theorem}

\begin{remark}\label{local harmonic radius lower bound remark} One should understand (\ref{injectivity radius explaination}) as follows:
for an open set $U$ inside a Riemannian manifold with boundary, $inj_U\geq i_0$ means for each $p\in U$, $\exp_{p}$ maps $B_{i_0}(0)\subset T_p M$ diffeomorphically onto its image if $d(p,U^c)\geq i_0$, and maps $B_{\frac{1}{2}d(p,U^c)}(0)\subset T_p M$ diffeomorphically onto its image if $d(p,U^c)\leq i_0$.
\end{remark}

\begin{proof}
If not, we have a sequence $(M_k,\tilde{g}_k)$ and $\Sigma_k$, $\Omega_k$ that satisfies the conditions, but there exists $p_k\in \Omega_k$ with
$$\frac{r_h^{m+1,\alpha}(p_k,\tilde{g}_k,Q)}{d_{\tilde{g}_k}(p_k,\partial^+ \Omega_k)}=\inf_{p\in \Omega_k} \frac{r_h^{m+1,\alpha}(p_k,\tilde{g}_k,Q)}{d_{\tilde{g}_k}(p,\partial^+ \Omega_k)}\rightarrow 0.$$  Rescale the metric $g_k=(r_h^{m+1,\alpha}(p_k,\tilde g_k,Q))^{-2}\tilde{g}_k$, so $r_h^{m+1,\alpha}(p_k,{g}_k,Q)=1$, then $d_{g_k}(p_k,\partial^+\Omega_k)\rightarrow\infty$, and $r_h^{m+1,\alpha}(p,{g}_k,Q)\geq\frac{1}{2}$ if $d_{g_k}(p,p_k)\leq R$ , $k\geq k(R)$. Fix any $\beta\in(0,\alpha)$. Then there are two cases: 
 
\textbf{Case 1}
 $d_{g_k}(p_k,\Sigma_k)\rightarrow\infty$ for some subsequence.

Then a subsequence $(M_k,{g_k},p_k)$ converges in pointed $C^{m+1,\beta}$ sense to a complete Riemannian manifold $(M_\infty,g_\infty,p_\infty)$.  So $\text{Ric}_{M_\infty}=0$, $inj_{M_\infty}=\infty$. By Cheeger-Gromoll splitting theorem, $(M_\infty, g_\infty)$ is isometric to $(\mathbb{R}^n, g_{flat})$. 

Hence for any $L>0$, there exist a coordinate $\varphi_{0,k}:B_{L+5} \rightarrow U_k\subset M_k$, $\varphi_{0,k}(0)=p_k$,  such that
$$\norm{g_{k,ij}-\delta_{ij}}_{C^{m+1,\beta}(B_{L+5})}\rightarrow 0, 1\leq i,j\leq n.$$ 
We solve the Dirichlet problem for functions $u_k^\nu$, $1\leq\nu\leq n$:

\begin{equation}
\Delta_{M_k}u_k^\nu=0 \ {\rm in} \ {B}_{L+5} , u_k^\nu|_{\partial {B}_{L+5}}=x^\nu\ .
\end{equation}
Recall the formula $$\Delta_{g}=g^{ij}\partial_i\partial_j+\frac{1}{\sqrt{|g|}}{\partial_i(\sqrt{|g|}g^{ij})}\partial_j, |g|=\det(g_{ij}).$$  
Then we have $$\lVert u_k^\nu-x^\nu\rVert_{C^{m+2,\beta}({B}_{L+5})}\leq C\lVert \Delta_{ M_k}(u_k^\nu-x^\nu)\rVert_{C^{m,\beta}({B}_{L+5})}\rightarrow 0.$$
Hence, we get a new coordinate system $(u_k^1,\cdots,u_k^n)$ and we discard the original coordinate system, and we use the same notation for tensors written in the new coordinate system, so in the new coordinate system we have
$$\lVert g_{k,ij}-\delta_{ij} \rVert_{C^{m+1,\beta}(B_{L+3})}\rightarrow 0.$$ 
Now we want to improve the convergence of $g_{k,ij}$     from elliptic equations.
We have a system of equations
$$\Delta_{M_k} g_{k,ij}+B_{ij}(g_k,\partial g_k)=-2\text{Ric}_{M_k,ij}. $$
where $B_{ij}(g,\partial g)$ are  polynormials of $g,\partial g$ and are quadratic in $\partial g$.
From $W^{m+2,p}$ estimates, Morrey embeddings, and $$|\nabla^l \text{Ric}_{ M_k}|\rightarrow 0, 0\leq l\leq m,$$
we have for $1\leq i,j\leq n$

\begin{equation}\label{harmonic coordinate estimates for use of lower dimensional}
    \begin{split}
\norm{g_{k,ij}-\delta_{ij}}_{C^{m+1,\alpha}( B_{L+2})}\leq& C(\norm{\Delta_{ M_k} (g_{k,ij}-\delta_{ij})}_{C^m( B_{L+3})}\\&+\norm{g_{k,ij}-\delta_{ij}}_{L^\infty(B_{L+3})})\rightarrow 0.
    \end{split}
\end{equation}
Hence we get a $(2(L+2), Q,m+1,\alpha)$ harmonic coordinate chart centered at $p_k'$, with $d_{g_k}(p_k',p_k)\rightarrow 0$ , then  $r_h^{m,\alpha}(p_k,{g}_k,Q)\geq 2(L+1)$ for large $k$, which is a contradiction. 

\textbf{Case 2}
$d_{{g}_k}(p_k,\Sigma_k)\leq K.$

A subsequence $(M_k,{g_k},p_k)$ converges in pointed $C^{m+1,\beta}$ sense to a complete Riemannian manifold with boundary $(M_\infty,g_\infty,p_\infty)$ and $(\partial M_k,{g}_k,q_k)$ converges in $C^{m+1,\beta}$ sense to $(\partial M_\infty,g_{\infty}|_{\partial M_\infty},q_\infty)$, where $q_k\in\Sigma_k$ is the unique foot point of $p_k$ in $\Sigma_k$. Then $\text{Ric}_{M_\infty}=0$, $\text{Ric}_{\partial M_\infty}=0$, $H_\infty=0$, $inj_{\partial M_\infty}=\infty$, $i_{b,M_\infty}=\infty$. Hence
$(\partial M_\infty, g_\infty|_{\partial M_\infty})$ is isometric to $(\mathbb{R}^{n-1}, g_{flat})$. By (\ref{scalar curvatures and hypersurface}), we have $S_\infty=0$. Then by Lemma $\ref{flat totally geodesic}$, $(M_\infty,g_\infty)$ is a smooth Riemannian manifold with boundary and $Rm_{M_\infty}=0$. Since also $i_{b,M_\infty}=\infty$, $(M_\infty, g_\infty)$ is a isometric to $(\mathbb{R}^n_+,g_{flat})$.  

Hence for any $L>2K+10$, there exist a coordinate $\varphi_{0,k}:B_{L+5}^+\rightarrow U_k\subset M_k$, $\varphi_{0,k}(0)=q_k$, $\varphi_{0,k}(\tilde{B}_{L+5})=U_k\cap\partial M_k$ such that
$$\norm{g_{k,ij}-\delta_{ij}}_{C^{m+1,\beta}(B_{L+5}^+)}\rightarrow 0, 1\leq i,j\leq n.$$  
First, we solve for functions $v_k^\nu$,     $1\leq\nu\leq n-1$, 
\begin{equation}
\Delta_{\partial M_k}v_k^\nu=0 \ {\rm in} \ \tilde{B}_{L+5} , v_k^\nu|_{\partial \tilde{B}_{L+5}}=x^\nu ,
\end{equation}
Then we have $$\lVert v_k^\nu-x^\nu\rVert_{C^{m+2,\beta}(\tilde{B}_{L+5})}\leq C\lVert \Delta_{\partial M_k}(v_k^\nu-x^\nu)\rVert_{C^{m,\beta}(\tilde{B}_{L+5})}\rightarrow 0.$$
Next, we solve for $1\leq\nu\leq n-1$,
\begin{equation}
\Delta_{M_k} u_k^\nu=0  \ {\rm in} \ B_{L+5}^+, u_k^\nu|_{\tilde{B}_{L+5}}=v_k^\nu, u_k^\nu|_{{\partial^+B_{L+5}^+}}=x^\nu.
\end{equation}
Note that $\partial B^+_{L+5}$ is not a $C^1$-boundary, but it satisfies exterior sphere condition, so we can solve the equations by Perron's method to get a unique solution  $u_k^\nu \in C^\infty((B_{L+5}^+)^\circ)\cap C^0(\overline{{B}_{L+5}^+}$).
From definitions and the estimates above, we have
$$
\lVert\Delta_{M_k}(u_k^\nu-x^\nu)\rVert_{C^{m,\beta}(B^+_{L+5})}\rightarrow 0,$$  $$\lVert u_k^\nu-x^\nu\rVert_{C^{m+2,\beta}(\tilde{B}_{L+5})}\rightarrow 0,$$ $$\lVert u_k^\nu-x^\nu\rVert_{L^\infty(\partial B^+_{L+5})}\rightarrow 0,$$
then by maximum principle, we have  $$\lVert u_k^\nu-x^\nu\rVert_{L^\infty(B^+_{L+5})}\rightarrow 0,$$ and by Schauder estimates

\begin{equation}
\begin{split}
\lVert u_k^\nu-x^\nu\rVert_{C^{m+2,\beta}(B_{L+4}^+)}\leq C(&\lVert \Delta_{M_k}(u_k^\nu-x^\nu)\rVert_{C^{m,\beta}(B_{L+5}^+)}+\lVert u_k^\nu-x^\nu\rVert_{L^\infty(B_{L+5}^+)}\\ 
&+\lVert u_k^\nu-x^\nu\rVert_{C^{m+2,\beta}(\tilde{B}_{L+5})})\rightarrow 0.
\end{split}
\end{equation}
Next, we construct $u_k^n$ by solving

$$\Delta_{M_k}u_k^n=0 \ {\rm in}\ B^+_{L+5},u_k^n|_{\partial B_{L+5}^+}=x^n.  $$ We have

\begin{equation}
    \begin{split}
        \lVert u_k^n-x^n\rVert_{C^{m+2,\beta}(B^+_{L+4})}\leq& C(\lVert \Delta_{M_k}(u_k^n-x^n)\rVert_{C^{m,\beta}(B_{L+5}^+)}\\&+\lVert u_k^n-x^n\rVert_{L^\infty(B_{L+5}^+)}) \rightarrow 0.
    \end{split}
\end{equation}
Hence we get a new coordinate system $(u_k^1,\cdots,u_k^n)$ and we discard the original coordinate system, and we use the same notation for tensors written in both coordinate systems, so in the new coordinate system we have
\begin{equation}\label{lower order convergence harmonic coordinates}
    \lVert g_{k,ij}-\delta_{ij} \rVert_{C^{m+1,\beta}(B_{L+3}^+)}\rightarrow 0.
\end{equation}
Now we want to improve the convergence of $g_{k,ij}$ from elliptic equations with Neumann boundary conditions.
We have equations
\begin{equation}\label{Ricci boundary equation}
\Delta_{\partial M_k} g_{k,ij}+\tilde{B}_{ij}(g_k,\partial g_k)=-2\text{Ric}_{\partial M_k,ij} 
\end{equation}
\begin{equation}\label{Ricci equation}
\Delta_{M_k} g_{k,ij}+B_{ij}(g_k,\partial g_k)=-2\text{Ric}_{M_k,ij} 
\end{equation}
Fix $\theta\in (\beta,1),  p=\frac{n}{1-\theta}$, from $W^{m+2,p}$ estimates, Morrey embeddings, and $$|\nabla_{\partial M_k}^l \text{Ric}_{\partial M_k}|\rightarrow 0, 0\leq l\leq m,$$
we have for $1\leq i,j\leq n-1$,
\begin{equation}\label{lower dimensional convergence harmonic coordinates}
\begin{split}
\norm{g_{k,ij}-\delta_{ij}}_{C^{m+1,\theta}(\tilde B_{L+2.5}^+)}\leq &C(\norm{\Delta_{\partial M_k} (g_{k,ij}-\delta_{ij})}_{C^m(\tilde B_{L+3}^+)}\\&+\norm{g_{k,ij}-\delta_{ij}}_{L^\infty(\tilde B_{L+3}^+)})\rightarrow 0.
\end{split}
\end{equation}
By Theorem 8.33 in \cite{gilbarg2015elliptic},
\begin{equation}
\begin{split}
\norm{g_{k,ij}-\delta_{ij}}_{C^{m+1,\theta}(B_{L+2}^+)}  \leq & C(\norm{\Delta_{ M_k} (g_{k,ij}-\delta_{ij})}_{C^m(\tilde B_{L+2.5}^+)}\\ &+\norm{g_{k,ij}-\delta_{ij}}_{L^\infty( B_{L+2.5}^+)}\\ &+\norm{g_{k,ij}-\delta_{ij}}_{C^{m+1,\theta}(\tilde B_{L+2.5}^+)})\rightarrow 0.
\end{split}
\end{equation}
Note that
\begin{equation}\label{Neumann boundary condition 1}
N_kg_k^{nn}=-2(n-1)H_kg_k^{nn},
\end{equation}
\begin{equation}\label{Neumann boundary condition 2}
N_kg_k^{in}=-(n-1)H_kg_k^{in}+\frac{1}{2\sqrt{g_k^{nn}}}g_k^{ij}\partial_j g_k^{nn},
\end{equation}
where $N_k=\frac{g_k^{jn}\partial _j}{\sqrt{g_k^{nn}}}$ is the unit normal vector of $\partial M_k$, $1\leq i\leq n-1$ and $j$ sums from $1$ to $n$, then we have Neumann boundary conditions for (\ref{Ricci equation}).  
For simplicity, assume for a while $m=0$. Since $$\norm{g_{k,ij}-\delta_{ij}}_{C^{1,\beta}(B_{L+2}^+)}\rightarrow 0 ,$$ $$|\text{Ric}_{M_k, ij}|_{C^0(B_{L+2}^+)}\rightarrow 0, 1\leq i,j\leq n,$$ $$|H_k|_{C^1(\tilde{B}_{L+2})}\rightarrow 0,$$ we have

$$ \norm{\Delta_{M_k} (g_k^{nn}-\delta^{nn})}_{C^0(B_{L+2}^+)}\rightarrow 0,   \norm{N_kg_k^{nn}}_{C^1(\tilde B_{L+2})}\rightarrow 0,$$
 then by Morrey embeddings (together with extensions), and $W^{2,p}$ estimates for Neumann boundary problems (for example, see a priori estimate 2.3.1.1 in \cite{grisvard2011elliptic}),

\begin{equation}
\begin{split}
\norm{g_k^{nn}-\delta^{nn}}_{C^{1,\theta}(B^+_{L+1.7})}\leq & C \norm{g_k^{nn}-\delta^{nn}}_{W^{2,p}(B_{L+1.8}^+)}\\
\leq & C( \norm{g_k^{nn}-\delta^{nn}}_{L^{p}(B_{L+2}^+)}+ \norm{\Delta_{M_k} (g_k^{nn}-\delta^{nn})}_{L^p(B_{L+2}^+)}\\ &+\norm{N_kg_k^{nn}}_{W^{1-\frac{1}{p},p}(\tilde B_{L+2})}\rightarrow 0.
\end{split}
\end{equation}
Now for $1\leq l\leq n-1$, since $$\norm{\Delta_{M_k}( g_k^{ln}-\delta^{ln})}_{C^0(B_{L+2}^+)}\rightarrow 0,$$ 
and
\begin{equation}
\begin{split}
\norm{N_kg_k^{ln}}_{W^{1-\frac{1}{p},p}(\tilde{B}_{L+1.5})} &\leq C(\norm{g_k^{nn}-\delta^{nn}}_{W^{2-\frac{1}{p},p}(\tilde{B}_{L+1.5})}+\norm{ H_k}_{W^{1-\frac{1}{p},p}(\tilde{B}_{L+1.5})})\\& \leq C(\norm{g_k^{nn}-\delta^{nn}}_{W^{2,p}({B}^+_{L+1.7})}+\norm{ H_k}_{C^1(\tilde{B}_{L+1.7})})\rightarrow 0.
\end{split}
\end{equation}
Then 
\begin{equation}
\begin{split}
\norm{g_k^{ln}-\delta^{ln}}_{C^{1,\theta}(B_{L+1.1}^+)}\leq & C\norm{g_k^{ln}-\delta^{ln}}_{W^{2,p}(B_{L+1.2}^+)}\\ \leq &C(\norm{g_k^{ln}-\delta^{ln}}_{L^{p}(B_{L+1.5}^+)}+\norm{\Delta_{M_k}( g_k^{ln}-\delta^{ln})}_{L^p(B_{L+1.5}^+)}\\  &+\norm{N_kg_k^{ln}}_{W^{1-\frac{1}{p},p}(\tilde{B}_{L+1.5})})\rightarrow 0.
\end{split}
\end{equation}
Hence $$\norm{g_{k,ij}-\delta_{ij}}_{C^{1,\theta}(B_{L+1.1}^+)}\rightarrow 0,  1\leq i,j\leq n.$$
For general $m\geq 1$,  take  $m$-th derivatives of (\ref{Ricci equation}) and the Neumann boundary conditions (\ref{Neumann boundary condition 1})(\ref{Neumann boundary condition 2}),
and note that $$[\partial _i,N_k]=(\frac{\partial_ig_k^{jn}}{\sqrt{g_k^{nn}}}-\frac{g_k^{jn}\partial_i g_k^{nn}}{2\sqrt{g_k^{nn}}^3})\partial_j,$$ so we get a system of second order elliptic equations with Neumann boundary conditions in $\partial_\gamma   g_k^{nn}$ and $\partial_\gamma g_k^{ln}$, $|\gamma|=m, 1\leq l\leq m-1$, with other terms freezed. Apply the previous estimates in the case $m=0$ and use (\ref{lower order convergence harmonic coordinates})(\ref{lower dimensional convergence harmonic coordinates}),
we get
$$\norm{g_k^{nn}-\delta^{nn}}_{C^{m+1,\theta}(B_{L+1}^+)}\rightarrow 0,$$
and then for $1\leq l\leq n-1,$
$$\norm{g_k^{ln}-\delta^{ln}}_{C^{m+1,\theta}(B_{L+1}^+)}\rightarrow 0,$$
hence  $$\norm{g_{k,ij}-\delta_{ij}}_{C^{m+1,\theta}(B_{L+1}^+)}\rightarrow 0,  1\leq i,j\leq n$$
In particular, take $\theta=\alpha$, one can  we get a $(\frac{L+1}{4}, Q,m+1,\alpha)$ harmonic coordinate chart centered at $p_k'$ , with $d_{g_k}(p_k',p_k)\rightarrow 0.$ Then  $r_h^{m+1,\alpha}(p_k,{g}_k,Q)\geq \frac{L}{4}$ for large $k$, which is a contradiction.

\end{proof}

\begin{remark}
Note that the case $m=0$ is also true, and one should be a little careful with the geometric arguments in the proof. Actually, the arguments in \cite{anderson2004boundary} prove a $C_*^{m+2}$ harmonic radius lower bound. 
\end{remark}

\begin{remark}\label{global harmonic radius lower bounds for complete manifolds with boundary}
The proof also shows that if $M$ is complete, $i_b\geq i_0$, $inj_M\geq i_0$, $inj_{\partial M}\geq i_0$ and (\ref{bounds 1 up to m})(\ref{bounds 2 up to m}) hold, then for any $p\in M$
\begin{equation}\label{harmonic radius lower bound globally}
     r^{m+1,\alpha}_h(p,g,Q)\geq r_0(i_0,\Lambda,m,\alpha,Q).
 \end{equation}

\end{remark}

The following corollary is a version we will use often.

\begin{corollary}\label{corollary that will be used often}
Let $(M_i,g_i)$ be a sequence of complete Einstein manifold with boundary. Suppose $i_b\geq i_0, inj_{\partial M}\geq i_0, |Rm|\leq C, |S|\leq C, |\nabla_{\partial M}^k Rm_{\partial M}|\leq C_k,|\nabla_{\partial M}^{k+1}H|\leq C_k, k\geq 0$, then for any $p_i\in M_i$,  there exists some subsequence such that  $(M_i,g_i,p_i)$ converges in pointed Cheeger-Gromov sense.
\end{corollary}

\begin{proof}
By Proposition \ref{volume noncollasping}, $|Rm|\leq C$, $|S|\leq C$, $i_b\geq i_0$,  together imply volume lower bounds of interiors balls of some fixed radius near boundary, hence also gives an interior injectivity radius lower bound from the following lemma. Then use Remark \ref{global harmonic radius lower bounds for complete manifolds with boundary} and Proposition \ref{harmonic radius lower bound}.
\end{proof}

The following lemma is well-known, which is a qualitative version of Theorem 4.3 in \cite{cheeger1982finite} and can also be easily proved by contradiction arguments.

\begin{lemma}\label{well-know interior inj lower bound} Let $(M,g)$ be a Riemannian manifold, and $B(p,r)$ be a metric ball that has compact closure. Suppose  $$\sup\limits_{B(p,r)}|Rm|\leq C, \emph{vol}(B(p,r))\geq v, $$then there exists $r_0>0$ depending on $n, C,v,r$ such that $\exp_q:B_{r_0}(0)\subset T_q M\rightarrow B(q,r_0)\subset M$ is a diffeomorphism for any $q\in B(p,\frac{r}{2})$.
\end{lemma}

\section{Convergence of hyperk\"ahler manifolds}\label{section main proof}

\subsection{Curvature estimates near the boundary}
This section serves as a first step for the proof of our main theorem.
For an Einstein manifold with boundary, if the boundary intrinsic and extrinsic geometry are controlled well and $i_b$ is bounded from below, we hope to control the interior geometry within $i_b$. To the author's knowledge, we do not know any general statement. We will first state and prove a version we need, and then discuss some lemmas needed in the proof.

\begin{theorem}\label{good boundary}
Let $(M,g)$ be a complete hyperk\"ahler 4-manifold with compact boundary. Suppose $|S|\leq C$,  $|\nabla^j_{\partial M} Rm_{\partial M}|\leq C_j, |\nabla^{j+1}_{\partial M} H|\leq C_j,    j=0,1,\cdots$, $inj_{\partial M}\geq i_0$, $i_b\geq i_0$, $\int_{M}|Rm|^2\leq C$. Then for any $r_1<i_0$, there exists $C'>0$, depending on $C,C_j,i_0,r_1$,  such that $\sup\limits_{N_{r_1}(\partial M,g)}|Rm|\leq C'$.

\end{theorem}

\begin{proof}
Without loss of generality, assume $i_0=1$. Denote $\alpha={r_1}$, $\beta=\frac{1}{4}(1-\alpha)$.  Suppose the conclusion is not true, we have a sequence $(M_i,g_i)$ satisfying the conditions, but $$\sup\limits_{N_{\alpha}(\partial M_i,g_i)}|Rm_{g_i}|\rightarrow\infty.$$ Let  $p_i\in N_{r_1}(\partial M_i,g_i)$  achieves this supremum. 

\textbf{Claim 1} There exists a subsequence such that 
\begin{equation}\label{Distance curvature}
d_{g_i}(p_i,\partial M_i)^2|Rm_{g_i}(p_i)|\rightarrow\infty.
\end{equation}
If this is not true, we have $\sup_i d_{g_i}(p_i,\partial M_i)^2 |Rm_{g_i}|(p_i)<\infty$. Rescale $\tilde{g}_i=|Rm_{g_i}(p_i)|g_i$, then $|Rm_{\tilde{g}_i}(p_i)|=1$, and $|Rm_{\tilde{g}_i}|\leq 1$ in $N_{\alpha|Rm(p_i)|^{\frac{1}{2}}}(\partial M_i,\tilde{g}_i)$, and $\sup_id_{\tilde{g}_i}(p_i,\partial M_i)<\infty$, $ i_{b,\tilde g_i}\geq |Rm_{g_i}(p_i)|^{\frac{1}{2}}$ for all $i$. Hence by Corollary \ref{corollary that will be used often}, $(M_i,g_i,p_i)$ subconverges in pointed Cheeger-Gromov sense to $(M_\infty, \tilde{g}_\infty,p_\infty)$, which is a complete Ricci-flat 4-manifold with flat, totally geodesic boundary, hence must be flat by Lemma \ref{flat totally geodesic}. This contradicts that $|Rm_{g_\infty}(p_\infty)|=1$ and proves Claim 1.

Now rescale $g_i$ in another way, let $g_i'=d_{g_i}(p_i,\partial M_i)^{-2}g_i$, so $d_{g_i'}(p_i,\partial M)=1$.
Since $d_{g_i}(p_i,\partial M_i)\leq \alpha$, the rescaled metric $g_i'$ satisfies $i_{b,g_i'}\geq \alpha^{-1}$ as well as all other conditions of the assumptions of the theorem, but with different bounds, regardness of whether $d_{g_i}(p_i,\partial M_i)$ is uniformly bounded from below or not. Moreover, (\ref{Distance curvature}) is equivalent to $|Rm_{g_i'}(p_i)|\rightarrow\infty$.

By the $\epsilon$-regularity Theorem \ref{Cheeger-Tian}, there exists a universal constant $\epsilon_0$ such that  for sufficiently large $i$, $$\int_{B_{{g}_i'}(p_i,\beta )}|Rm_{g_i'}|^2\geq\epsilon_0.$$

\textbf{Claim 2} There exists a subsequence such that
$$ \sup\limits_{ N_{\alpha}(\partial M_i,g_i')}|Rm_{g_i'}|\rightarrow \infty, $$
If not, we have
$ \sup\limits_{ N_{\alpha}(\partial M_i,g_i')}|Rm_{g_i'}|\leq C$. By Lemma \ref{volume noncollasping} and Bishop-Gromov volume comparison, $\text{vol}(B_{g_i'}(p_i,\beta))\geq v$. Since $B_{g_i'}(p_i,\beta)\subset N_{\alpha^{-1}(\alpha+\beta)}(\partial M_i,g_i')\subset N_{\alpha+\beta}(\partial M_i,g_i)$, and the last one is diffeomorphic to $\partial M_i\times [0,\alpha+\beta]$, we conclude that there is no $-2$ curve in $B_{g_i'}(p_i,\beta)$. By Proposition \ref{no bubble},
$|Rm_{g_i'}(p_i)|$ is bounded, which is a contradiction to Claim 1 and finishes the proof of Claim 2.

Now Claim 2 enables us to get by induction, for each fixed positive integer $N$, $N$ sequences of metrics $g_i^{(0)}=g_i,g_i^{(1)}=g_i',\cdots, g_i^{(N)}$, and points $p_i^{(j)}\in N_{\alpha}(\partial M_i,g_i^{(j)})$ for  $0\leq j\leq N-1$, $p_i^{(0)}=p_i$, such that for $0\leq j\leq N-1$, $p_i^{(j)}$ achieves the supremum of $|Rm_{g_i^{(j)}}|$ in $N_{\alpha}(\partial M_i,g_i^{(j)})$, and

 $$|Rm_{g_i^{(j)}}(p_i)|\rightarrow\infty,$$  $$d_{g_i^{(j+1)}}(p_i^{(j)},\partial M_i)=1,$$
$$\int_{B_{{g}_i^{(j+1)}}(p_i^{(j)},\beta)}|Rm_{g_i^{(j+1)}}|^2\geq\epsilon_0,$$

$$ B_{g_i^{(j+1)}}(p_i^{(j)},\beta)\subset N_{\alpha^{-1}(\alpha+\beta)}(\partial M_i,g_i^{(j+1)})\subset N_{\alpha+\beta}(\partial M_i,g_i^{(j)}) ,$$ 
 $$B_{g_i^{(j+1)}}(p_i^{(j)},\beta)\cap N_{\alpha+\beta}(\partial M_i,g_i^{(j+1)})=\emptyset . $$               
It follows that for each fixed $i$, $B_{g_i^{(j+1)}}(p_i^{(j)},\beta)$ does not interect each other for different $j$. Since $\int_{M_i}|Rm_{g_i}|^2\leq C$, we have $N\epsilon_0\leq C$. This is a contradiction, since $N$ can be any positive integer.
\end{proof}

\begin{remark}
This theorem is purely local. In fact, by slightly modifying the proof, we see that if the bounds in the assumptions hold in a metric cyclinder $C(B_{\partial M}(p,r_0),0,r_1)$  such that $\exp^\perp$ maps $B_{\partial M}(p,r_0)\times [0,r_1)$ diffeomorphically onto it, and such that $B_{\partial M}(p,r_0)$ has compact closure,  then we have curvature bounds in any interior metric cyclinder $C(B_{\partial M}(p,r_0'),0,r_1')$ with fixed $r_0'<r_0,r_1'<r_1$.
\end{remark}

The following collapsing $\epsilon$-regularity theorem is originally due to Cheeger-Tian in \cite{cheeger2006curvature}. Recently, in the hyperk\"ahler case, \cite{sun2021collapsing} gives a simple proof by a blow-up argument and studying complete collapsing limits of hyperk\"ahler manifolds with bounded curvature.

\begin{proposition}\label{Cheeger-Tian} There exists $\epsilon,c$ such that the following holds:
Let $(M^4,g)$ be an Einstein 4-manifold, $|\emph{Ric}|\leq 3$, $r\leq 1$, and $B(p,r)$ is a metric ball that has compact closure. If $$\int_{B(p,r)}|Rm|^2\leq\epsilon,$$ then $$\sup\limits_{B(p,\frac{r}{2})}|Rm|\leq c r^{-2}.$$
\end{proposition}

The following proposition plays an important role. To avoid redundancy, we state it here without proving, but we will prove a more general version later, see Proposition \ref{no bubble 2}.

\begin{proposition}\label{no bubble}
Let $(M,g)$ be an hyperk\"ahler 4-manifold. Suppose $B(p,5)$ has compact closure, $B(p,3)$
contains no $-2$ curve,  $$\emph{vol}(B(p,1))\geq v,$$
$$\int_{B(p,3)}|Rm|^2\leq C.$$Then there exists $C'>0$, depending on $v,C$ such that $$\sup_{B(p,1)}|Rm|\leq C'.$$
\end{proposition}

The following lemma originally dates back to Koiso in \cite{koiso1981hypersurfaces}, 
and is an incredibly special case of the result in \cite{biquard:hal-02928859}\cite{anderson2008unique}.
Since it plays an important role throughout the paper, we provide a detailed proof here following \cite{koiso1981hypersurfaces}.

\begin{lemma}\label{flat totally geodesic}
Let $(M,g)$ be a connected  $C^2$ Riemannian manifold with boundary. Suppose $\emph{Ric}_M=0$ and for some open boundary portion $T$, $S|_T=0$ and $Rm_{\partial M}|_T=0$, then $g$ is smooth and $Rm_M=0$.
\end{lemma}

\begin{proof}
Fix any point $p$ in $T$. First, we show $g$ can be extended across the boundary near $p$.
Choose a semi-geodesic coordinate system $(x^1,\cdots, x^n)$ near $p$, with $\partial M$ identified with $\{x^n=0\}$, and interior identified with $\{x^n>0\}$, $\nabla x^n=\partial_{ x^n}$, and $g_{ij}(x',0)=\delta_{ij}, 1\leq i,j\leq n-1$, where $x'=(x^1,\cdots, x^{n-1})$.  We extend the metric tensor $g$ by reflection across $\{x^n=0\}$, i.e., set $g_{ij}(x',x^n)=g_{ij}(x',-x^n), 1\leq i,j\leq n$. Then $g\in C^0(B)\cap C^2(B^+)$, where $B^+$ is the boundary coordinate ball and $B$ is its extension after reflection. We need to show $g\in C^2(B)$. In fact, $S=0$ is equivalent to $\frac{\partial g_{ij}}{\partial x^n_+}(x',0)=0, 1\leq i,j\leq n-1,$ then we also have $\frac{\partial g_{ij}}{\partial x^n_-}(x',0)=-\frac{\partial g_{ij}}{\partial x^n_+}(x',0)=0,$ hence  $\frac{\partial g_{ij}}{\partial x^n}(x',0)=0$ and $g_{ij}\in C^1(B)$. Since $g\in C^2(B^+)$, we have for $1\leq l\leq n-1$,
 $$\frac{\partial }{\partial x^n_+}\frac{\partial g_{ij}}{\partial x^l}(x',0)=\frac{\partial }{\partial x^l}\frac{\partial g_{ij}}{\partial x^n_+}(x',0)=\frac{\partial }{\partial x^l}0=0,$$ $$\frac{\partial }{\partial x^n_-}\frac{\partial g_{ij}}{\partial x^l}(x',0)=-\frac{\partial }{\partial x^n_+}\frac{\partial g_{ij}}{\partial x^l}(x',0)=0.$$ 
 \begin{equation}
 \begin{split}
      \frac{\partial}{\partial x^n_-}\frac{\partial g_{ij}}{\partial x^n}(x',0)&=\lim\limits_{x^n\rightarrow 0^-} \frac{1}{x^n}\frac{\partial g_{ij}}{\partial x^n}(x',-x^n)\\ &=\lim\limits_{x^n\rightarrow 0^+}\frac{1}{x^n}\frac{\partial g_{ij}}{\partial x^n}(x',x^n)=\frac{\partial}{\partial x^n_+}\frac{\partial g_{ij}}{\partial x_n}(x',0).
 \end{split}
 \end{equation}
  Hence $g_{ij}\in C^2(B)$. Finally, we have $g_{nn}=1,g_{ln}=g_{nl}=0, 1\leq l\leq n-1$, hence $g\in C^2(B)$.

By elliptic regularity, all harmonic coordinate charts in $B$ give rise to a real analytic structure in $B$ such that $g$ is real analytic. Hence if $t$ is a distance function such that $\partial M$ is defined by $t^{-1}(0)$ near $p$, then $t$ is real analytic near $p$. Choose a real analytic coordinate $(z,t)$ near $p$. Since $t^{-1}(0)$ is totally geodesic, $\frac{\partial g}{\partial t}(z,0) =0$. The evolution equation (\ref{evolution equation for 2nd form}) is equivalent to the second order PDE
\begin{equation}
\frac{\partial^2 g}{\partial t^2}=2 ric \ g-\frac{1}{2}tr_g(\frac{\partial g}{\partial t})\frac{\partial g}{\partial t}+(\frac{\partial g}{\partial t})^2,
\end{equation}
 where $ric\ g$ is the Ricci tensor of level sets of $t$. By the uniqueness part of Cauchy-Kovalevskaya theorem, we know $g(z,t)=g(z,0)$. Hence $Rm_M=0$ near $p$. Since $Rm_M$ is real analytic in the interior of $M$, $Rm_M=0$ in $M$.

\end{proof}

\subsection{Convergence of hyperk\"ahler metrics}

\begin{theorem}\label{main theorem hyperkahler metric}
Let $(X_i,g_i)$ be a sequence of compact, connected hyperk\"ahler 4-manifold with boundary, suppose on $\partial X_i$, we have
$$H_i\geq H_0>0, |S_i|\leq C, |\nabla_{\partial X_i}^{j+1} H|\leq C_j,$$   $$inj_{\partial X_i}\geq i_0, \emph{diam}_{g_i|_{\partial X_i}}(\partial X_i)\leq C,|\nabla^j_{\partial X_i}Rm_{\partial X_i}|\leq C_j,\forall j\geq 0, $$ 
and $\chi(X_i)\leq C$.
Assuming there exists no $-2$ curve on $X_i$, 
then there exists a subsequence such that $(X_i,g_i)$ converges in Cheeger-Gromov sense to a compact, connected hyperk\"ahler 4-manifold with boundary $(X_\infty, g_\infty).$
\end{theorem}

\begin{remark}
By Chern-Gauss-Bonnet formula, our assumptions imply $$\int_{X_i}|Rm|^2\leq C.$$ Note that for a compact connected Einstein 4-manifold $(M,g)$ with boundary, the Chern-Gauss-Bonnet formula says
$$\frac{1}{8\pi^2}\int_M|Rm|^2=\chi(M)-\frac{1}{2\pi^2}\int_{\partial M}\prod\limits_{i=1}^3\lambda_i-\frac{1}{8\pi^2}\int_{\partial M}\sum\limits_{\sigma\in S_3}K_{\sigma_1\sigma_2}\lambda_{\sigma_3}. $$
Here $\lambda_1,\lambda_2,\lambda_3$ are eigenvalues of the shape operator $S$ of $\partial M$, Let $e_i$ be eigenvectors of eigenvalue $\lambda_i$ such that $\{e_1,e_2,e_3\}$ is an orthonormal basis, then $K_{ij}=\sec(e_i,e_j) $. See for example (1.16) in \cite{Anderson2000L2CA}.
\end{remark}

\begin{remark}
If we drop the condition $\text{diam}_{g_i|_{\partial X_i}}(\partial X_i)\leq C$, and replace $H_i\geq H_0>0$ by $H_i>0$, $\chi(X_i)\leq C$ by $\int_{X_i}|Rm_{g_i}|^2\leq C$, then for any point $p_i\in X_i$, a subsequence of $(X_i,g_i,p_i)$ converges in pointed Cheeger-Gromov sense to a complete, connected hyperk\"ahler 4-manifolds with nonempty or empty boundary $(X_\infty,g_\infty,p_\infty)$, depending on whether the distance of $p_i$ to boundary is bounded or not. 
\end{remark}

\begin{remark}\label{positive mean curvature essential}
The positive mean curvature condition is necessary. The following counterexample is natural and was observed by Donaldson in \cite{donaldson2018remarks}. Consider the standard unit ball $B^4$  inside Euclidean $\mathbb{R}^4$, ``squeeze'' the ball such that the north pole and the south pole of the boundary $S^3$ comes together, so we get a sequence of embedded $B^4$ in $\mathbb{R}^4$ converging in Hausdorff sense to a limit homeomorphic to the wedge sum of two $B^4$, whose boundary is an immersed $S^3$ intersecting itself at one point. For this sequence, all other assumptions are satisfied. Slightly modifying the process, one can also have a sequence of $B^4$ of dumbbell shape such that the middle cyclinder $B^3\times [0,1]$ collapses to $[0,1]$, then they have a Hausdorff limit which is homeomorphic to two $B^4$ joint by a line segment.

In these types of examples, the curvatures are uniformly bounded, and the global volume are non-collapsing before taking the limit.

\end{remark}

\begin{remark}\label{no -2 curve essential}

If we allow $-2$ curves and do not assume the positive mean curvature condition, something worse will happen: consider the Kummer construction. Let $T^4/\mathbb{Z}_2$ be the flat 4-orbifold with 16 singularities, remove small neighborhoods of the 16 singularities, and glue in 16 copies of $T^*S^2$. By varing the sizes of these glue-in regions and perturbing the metrics, we get a sequence of hyperk\"ahler 4-manifolds $(M_i,h_i)$, each of which is diffeomorphic to a $K3$ surface, such that $(M_i,h_i)$ converges in Gromov-Hausdorff sense to $T^4/\mathbb{Z}_2$, and converges in Cheeger-Gromov sense away from these 16 singularities. Now let $T^3/\mathbb{Z}_2\subset T^4/\mathbb{Z}_2$ be the flat 3-orbifold, such that the last coordinate equal to 0. Let $\tilde{X}_i\subset T^4/\mathbb{Z}_2$  be the closure of the tubular neighborhood of $T^3/\mathbb{Z}_2$ of width $i^{-1}$, then $\partial \tilde{X}_i$ is connected, totally geodesic, isometric to the same flat $T^3$. Now for each $i$, choose $n(i)$ large enough such that one can find  $X_i\subset M_{n(i)}$ which are compact domains with smooth boundary,  $d_{GH}(X_i,\tilde{X}_i)\rightarrow 0$, $|\nabla_{\partial X_i}^j S_{\partial X_i}|\rightarrow 0 $, $\forall j\geq 0$, $\partial X_i$ converges in Cheeger-Gromov sense to $\partial \tilde X_i$. Let $g_i=h_{n(i)}|_{ X_i}$, then $(X_i,g_i)$ converges in Gromov-Hausdorff sense to flat $T^3/\mathbb{Z}_2$. In this case, $\text{vol}(X_i)\rightarrow 0$, $\sup\limits_{X_i}|Rm_{g_i}|\rightarrow\infty$.

Note that in this example $\partial \tilde{X}_i$ cannot be perturbed in flat $T^4/\mathbb{Z}_2$ to have positive mean curvature. In fact,
take a small tubular neighborhood of $\partial \tilde{X}_i$ of width less than $i^{-1}$, whose boundary has two totally geodesic connected components $T_1,T_2$. Suppose $\partial \tilde{X}_i$ can be perturbed to $T'$ such that its mean curvature has a strict sign, say that its mean curvature vector points towards $T_1$. Then $T', T_1$ bound a region $W$. By \cite{donaldson2018remarks} Proposition 7, $T'$ and $T_1$ are isometric, so $T'$ is totally geodesic, which is a contradiction.  

\end{remark}
The major step of the proof in \ref{main theorem hyperkahler metric} is to show that these Riemannian manifolds have a nice neighborhood of definite size. We show that the boundary injectivity radius has a lower bound, so that the interior geometry within the boundary injectivity radius is nicely controlled by Theorem \ref{good boundary}.

\begin{proposition}\label{global inj lower bound}
There exists $i_1>0$, depending on the constants in Theorem \ref{main theorem hyperkahler metric} such that    $i_{b,g_i}\geq i_1.$
\end{proposition}

\begin{proof}
Suppose not, we have a subsequence of hyperk\"ahler metrics $g_i$,  with $i_{b,{g}_i}\rightarrow 0$. Rescale the metric $\tilde g_i=i_{b,g_i}^{-2}{g}_i$, then $i_{b,\tilde g_i}=1$. For any point $p_i\in \partial X_i$, the restriction metric $(\partial X_i,\tilde{g}_i|_{\partial X_i},p_i)$ converges in pointed Cheeger-Gromov sense to  flat $\mathbb{R}^3$, $|S_{\tilde{g}_i}|\rightarrow 0$ and $|\nabla_{\partial X_i}^{j+1} H_i|\rightarrow 0$ uniformly on $\partial X_i$. Consider $\sup\limits_{B_{\tilde g_i}(p_i,4)}|Rm_{\tilde g_i}|$. We have two cases 

\textbf{Case 1} $\sup\limits_{B_{\tilde g_i}(p_i,4)}|Rm_{\tilde g_i}|\leq C.$\\
We have a subsequence  $(B_{\tilde g_i}(p_i,3),\tilde g_i)$ converges in Cheeger-Gromov sense to a Riemannian manifold with boundary  $(B_\infty,g_\infty)$, so $(B_\infty,g_\infty)$ is Ricci flat, and all boundary components are flat, totally geodesic. By Lemma \ref{flat totally geodesic}, $(B_\infty,g_\infty)$ is flat. We need to choose good points $p_i$ to lead to a contradiction. In fact, by Proposition \ref{existence of focal points}, there exists $p_i\in\partial X_i$ such that $\gamma_{p_i}(1)$ is a focal point along $\gamma_{p_i}$. Let $p_\infty$ be the limit of $p_i$. By Proposition \ref{pass to limit}, we get a limit geodesic $\gamma_{p_\infty}: [0,1]\rightarrow B_\infty$, such that $\gamma_{p_\infty}(1)$ is a focal point along $\gamma_{p_\infty}$, contradiction. 

\textbf{Case 2} For some subsequence we have $\sup\limits_{B_{\tilde g_i}(p_i,4)}|Rm_{\tilde g_i}| \rightarrow\infty$.\\
Then we can find points $q_i\in B_{\tilde{g}_i}(p_i,4)$ such that $|Rm_{\tilde g_i}(q_i)|\rightarrow\infty$. By Theorem \ref{good boundary}, we have $d_{\tilde g_i}(q_i,\partial X_i)\geq \frac{1}{2}$ for large $i$. By Theorem \ref{good boundary}, Proposition \ref{volume noncollasping}, and the fact  $d_{\tilde{g}_i}(q_i,\partial X_i)\leq 4$, we have  $\text{vol}_{\tilde{g}_i}(B_{\tilde{g}_i}(q_i,\frac{r_0}{2}))\geq v_0$ for some $r_0,v_0$. Since there is no $-2$ curve in $X_i$, then by Proposition \ref{no bubble}, $\sup_{B_{\tilde{g}_i}(q_i,\frac{r_0}{10})}|Rm_{\tilde{g}_i}|\leq C$, which is a contradiction.

\end{proof}

Then we can finish the proof of Theorem \ref{main theorem hyperkahler metric} as follows: by Theorem \ref{good boundary}, $|Rm_{g_i}|$ is uniformly bounded in $N_{\frac{i_1}{2}}(\partial X_i,g_i)$. By Proposition $\ref{volume noncollasping}$, there exist $r_1<\frac{i_1}{10},v_1$, such that $\text{vol}_{g_i}(B_{g_i}(p,r_1))\geq v_1$ for any $p$ with $d_{g_i}(p,\partial X_i)=2r_1$.  By Proposition \ref{diam}, $\sup\limits_{q\in X_i}d_{g_i}(q,\partial X_i)\leq 3H_0^{-1}$, hence $\text{diam} (X_i,g_i)\leq C$. Then from Bishop-Gromov volume comparison,  we know $\text{vol}_{g_i}(B_{g_i}(p,r_1))\geq v_2$ for any $p$ with $d_{g_i}(p,\partial X_i)\geq 5r_1$.  By Proposition $\ref{no bubble}$, for these $p$,
$|Rm_{g_i}(p)|$ is uniformly bounded, with the bound independent of $i$ and $p$. Hence $\sup_X |Rm_{g_i}|$ is uniformly bounded. Then by Corollary \ref{corollary that will be used often},
 a subsequence $(X_i,g_i)$ converges in Cheeger-Gromov sense to a smooth Riemannian manifold with boundary $(X,g)$, so $g$ is a hyperk\"ahler metric.

\begin{remark} One can also directly prove Theorem \ref{main theorem hyperkahler metric} by rescaling the maximum curvature norm to be 1. Suppose it is achieved at $p_i$. Then from Corollary \ref{Ric positive, mean positive, inj lower bound}, we know $i_b\geq i_0$ for the rescaled metrics. If now $d(p_i,\partial X_i)$ is bounded, then the pointed Riemannian manifolds converge in pointed Cheeger-Gromov sense to a Ricci-flat manifold with flat and totally geodesic boundary, hence the limit is flat, contradition. Otherwise, for a subsequence $d(p_i,\partial X_i)\rightarrow\infty$, then we rescale the metrics again such that this distance is 1. However, it is unknown whether the curvature is bounded in a fixed size neighborhood of the boundary in this scale. Using the idea of Theorem \ref{good boundary}, we can keep rescaling such that this will happen at some stage, for possibly different points $p_i'$, then get a contradiction. The contradiction is the same the one in Case 2 in Proposition \ref{global inj lower bound}, since volume lower bounds can be passed within finite distance, so do the curvature bounds by Proposition \ref{no bubble}. Alternatively, one can prove Theorem \ref{main theorem hyperkahler metric} by rescaling the harmonic radius, and all these methods turn out to be equivalent eventually.
\end{remark}

\subsection{Convergence of hyperk\"ahler triples}\label{moduli space}

Now we turn to the setting of \ref{convergence of triples}. ${X}$ will denote a compact oriented 4-manifold with boundary $\partial X=Y$. Let $\mathcal{N}^+$ be the set of closed framings $\bm{\gamma}=(\gamma_1,\gamma_2,\gamma_3)$ on $Y$  that satisfies the ``positive mean curvature'' condition

\begin{equation}\label{mean positive}
\sum\limits_{i=1}^3 \langle\gamma_i,d(*_{\bm\gamma}\gamma_i)\rangle_{\bm\gamma}>0.
\end{equation}
Let $\mathcal{M}^+$ be the set of smooth hyperk\"ahler triples $\bm\omega=(\omega_1,\omega_2,\omega_3)$ whose restriction to the boundary lies in $\mathcal{N}^+$, so we have a restriction map $p_0:\mathcal{M}^+\rightarrow\mathcal{N}^+$, which induce a map
\begin{equation}\label{map between moduli space}
    p:\mathcal{M}^+/\mathcal{G}_X\rightarrow\mathcal{N}^+/\mathcal{G}_Y.
\end{equation}
  Here $\mathcal{M}^+$, $\mathcal{N}^+$ are equipped with Fr\'echet topology defined by smooth convergence, $\mathcal{G}_X,\mathcal{G}_Y$ are orientation preserving diffeomorphism groups of $X,Y$, respectively. It is obvious that $p_0$, $p$ are continuous. 
Then the compactness part of Theorem \ref{convergence of triples} is equivalent to:
\begin{proposition}
When there is no $-2$ curve in $X$, the map $p:\mathcal{M}^+/\mathcal{G}_X\rightarrow\mathcal{N}^+/\mathcal{G}_Y$   is proper.
\end{proposition}

\begin{proof}
Suppose for some $\phi_i\in\mathcal{G}_Y$, and $\bm\gamma_i\in\mathcal{N}^+$,we have          $\phi_i^*{\bm\gamma_i}\rightarrow \bm\gamma\in \mathcal{N}^+$ and there exist $\bm\omega_i\in\mathcal{M}^+$ with $\bm\omega_i|_Y=\bm\gamma_i$, we want to show there exists $\psi_i\in\mathcal{G}_X$ and $\bm\omega\in\mathcal{M}^+$  such that $\psi_i^*\bm\omega_i\rightarrow\bm\omega$. Let $g_i$ be the Riemannian metric defined by $\bm\omega_i$. By the assumptions, we have a uniform positive lower bound for the mean curvature $H_i$ of $Y$ of the metric $g_{i}=g_{\bm\omega_i}$, and bounds for $|\nabla_{Y}^l S_i|$ for all $l\geq 0$, and $(Y,g_i|_{Y})$ converges in Cheeger-Gromov sense. Hence $(X,g_i)$ satisfies all conditions in Theorem \ref{main theorem hyperkahler metric}. Then for a subsequence, there exists diffeomorphism $\psi_i: X\rightarrow X$ such that $\psi_i^*g_i\rightarrow g$ smoothly as tensors. One can assume $\psi_i$ is orientation preserving, otherwise, for a subsequence, compose them with a fixed orientataion reversing diffeomorphism of $X$. Since $|\psi_i^*\bm\omega_i|^2=3$, $\psi_i^*\bm\omega_i$ are parallel, we conclude that for some subsequence $\psi_i^*\bm\omega_i\rightarrow\bm\omega$ smoothly, and $\bm\omega\in\mathcal{M}^+$.

\end{proof}

\subsection{Enhancements}

For the proof of the compactness part of Theorem \ref{convergence of triples enhancements}, one only needs the following proposition in place of Proposition \ref{no bubble}. Then we argue in the same way as the proof of Theorem \ref{convergence of triples}.

\begin{proposition}\label{no bubble 2}
Let $(M,g,\bm\omega)$ be a hyperk\"ahler 4-manifold. Suppose $B(p,5)$ has compact closure, and for any $-2$ curve $\Sigma$ in $B(p,3)$
, $$\Big{|}\int_{\Sigma}\bm\omega\Big{|}\geq a>0,$$  $$\emph{vol}(B(p,1))\geq v,$$ 
$$\int_{B(p,3)}|Rm|^2\leq C. $$ Then there exists $C'>0$, depending on $a,v,C$ such that $$\sup_{B(p,1)}|Rm|\leq C'.$$
\end{proposition}

\begin{proof}

For all $q\in B(p,2)$, by volume comparison, we have $$\text{vol}(B(q,1))\geq 3^{-4} \text{vol}(B(q,3))\geq 3^{-4}\text{vol}(B(p,1))\geq 3^{-4} v.$$
Suppose the conclusion is not true, then we have a sequence $(M_i,g_i,p_i)$ satisfies the conditions, but
 there exists $q_i'\in B(p_i,1)$ with $|Rm_{g_i}(q_i')|\rightarrow\infty$. By the following Lemma \ref{point selection}, we can find points $q_i\in B(p_i,2)$ such that $|Rm_{g_i}(q_i)|\geq |Rm_{g_i}(q_i')|$, and $$\sup\limits_{B_{g_i}(q_i,|Rm_{g_i}(q_i')|^{\frac{1}{2}}|Rm_{g_i}(q_i)|^{-\frac{1}{2}})}|Rm_{g_i}|\leq 4|Rm_{g_i}(q_i)|.$$  Rescale the metric $\tilde{g}_i=|Rm_{g_i}(q_i)|g_i$, $\tilde{\bm\omega}_i=|Rm_{g_i}(q_i)|\bm\omega_i$, so $\tilde{\bm\omega}_i$ defines $\tilde{g}_i$. Then we have $$\sup\limits_{B_{\tilde{g}_i}(q_i,|Rm_{g_i}(q_i')|^{\frac{1}{2}})}|Rm_{\tilde{g}_i}|\leq 4 , $$  $$|Rm_{\tilde{g}_i}(q_i)|=1,$$ $$ \text{vol}_{\tilde{g}_i}(B_{\tilde{g}_i}(q_i,r))\geq 3^{-4}vr^n,\forall r\leq|Rm_{g_i}(q_i)|^{\frac{1}{2}},$$ and
$$\int_{B_{\tilde{g}_i}(q_i, |Rm_{g_i}(q_i)|^{\frac{1}{2}})}|Rm_{\tilde{g}_i}|^2\leq \int_{B_{\tilde{g}_i}(p_i, 3|Rm_{g_i}(q_i)|^{\frac{1}{2}})}|Rm_{\tilde{g}_i}|^2\leq C.$$
Hence, for a subsequence, $(M_i,\tilde{g}_i,q_i)$ converges in Cheeger-Gromov topology to a complete non-flat hyperk\"ahler 4-manifold  $(M_\infty,g_\infty,q_\infty)$ with maximum volume growth and $$\int_{M_\infty}|Rm_{g_\infty}|^2\leq C.$$
Since $\tilde{\bm\omega}_i$ are parallel, and $|\tilde{\bm\omega}_i|^2=6$, $\tilde{\bm\omega}_i$ also subconverges to a hyperk\"ahler triple $\bm\omega_\infty$ that defines $g_\infty$. By \cite{bando1989construction}, $(M_\infty,g_\infty)$ is a hyperk\"ahler ALE space of order 4. Hence, from Kronheimer's classification \cite{kronheimer1989construction}\cite{kronheimer1989torelli}, there exists a $-2$ curve $C_\infty$ in $B_{g_\infty}(q_\infty, R)$ for some $R>0$
such that $|\int_{C_\infty}\bm\omega_{\infty}|\neq 0$.
Let $\phi_i:B_{g_\infty}(q_\infty,R)\rightarrow V_i$ be diffeomorphisms, such that $\phi_i^*\tilde{\bm\omega}_i\rightarrow \bm\omega_\infty$ smoothly. Let $C_i=(\phi_i)_*C_\infty$, then $C_i$ is a $-2$ curve in $B_{g_i}(p_i,3)$ for large $i$ and $$\int_{C_i} \bm\omega_i=|Rm_{g_i}(q_i)|^{-1}\int_{C_i}\tilde{\bm\omega}_i=|Rm_{g_i}(q_i)|^{-1}\int_{C_\infty} \phi_i^*\tilde{\bm\omega}_i \rightarrow 0, $$
which contradicts our assumption.

\end{proof}

The following point selection lemma is well-known and elementary. 

\begin{lemma}\label{point selection}
Let $(M,g)$ be a Riemannian manifold. Suppose $\sup\limits_{B(p,2)}|Rm|<\infty$, $|Rm(p)|\neq 0$. Then there exists a point $q\in B(p,2)$ such that $|Rm(q)|\geq |Rm(p)|$, and $\sup\limits_{B(q,{|Rm(p)|}^{\frac{1}{2}}{|Rm(q)|^{-\frac{1}{2}}})}|Rm|\leq  4|Rm(q)|.$  
\end{lemma}

\begin{proof}
If this is not true, let $A=|Rm(p)|^{\frac{1}{2}}$, then there exist $q_1\in B(p,2)$ such that $d(q_1,p)< 1$ and $|Rm(q_1)|> 4|Rm(p)|=4A^2$. By induction, we can find a sequence of points $q_0=p, q_1,q_2,\cdots$ with $d(q_{j+1},q_j)< |Rm(q_{j})|^{-\frac{1}{2}}A$, $d(q_{j+1},p)< 2-2^{-j}$ and $|Rm(q_{j+1})|>  4^{j+1} A^2$. This is an obvious contradiction since $\sup\limits_{B(p,2)}|Rm|<\infty.$
\end{proof}

\subsection{Uniqueness}\label{uniqueness 1}

Given the compactness results, it is natural to ask whether a subsequential limit $\bm\omega$ is unique(up to a diffeomorphism) in Theorem \ref{convergence of triples}, Theorem \ref{convergence of triples enhancements}. The answer is yes. In fact, both\cite{biquard:hal-02928859} and \cite{anderson2008unique} proved a
unique continuation theorem for Einstein metrics with prescribed boundary metric and second fundamental form, which implies the following:

\begin{proposition}\label{local uniqueness hyperkahler triple}
Let $X$ be a connected oriented 4-manifold with boundary. Suppose $\bm\omega_1,\bm\omega_2$ are two smooth hyperk\"ahler triples on $X$. Suppose $\bm\omega_1|_{\partial X}=\bm\omega_2|_{\partial X}$ , then in geodesic gauges of $g_{\bm\omega_1}$, $g_{\bm\omega_2}$, we have
$\bm\omega_1=\bm\omega_2$ near $\partial X$. 
\end{proposition}

\begin{proof} 
We have $g_{\bm\omega_1}|_{\partial X}=g_{\bm\omega_2}|_{\partial X}$ and $II_{g_{\bm\omega_1}}=II_{g_{\bm\omega_2}}$ on $\partial X$. By \cite{biquard:hal-02928859} Theorem 4, in geodesic gauges, we have $g_{\bm\omega_1}=g_{\bm\omega_2}$. Since $\nabla^{g_{\bm\omega_1}} |\bm\omega_1-\bm\omega_2|_{g_{\bm\omega_1}}^2=0$, and $\bm\omega_1=\bm\omega_2$ at one point, we have $\bm\omega_1=\bm\omega_2$ everywhere near $\partial X$.
\end{proof}

Now the following global uniqueness result follows from an analytic continuation argument (See \cite{kobayashi1963foundations}  Chapter VI, Section 6):

\begin{theorem}\label{unique continuation hyperkahler}
Let $X$ be a connected 4-manifold with boundary,  $\pi_1(X,\partial X)=0$. Suppose $\bm\omega_1,\bm\omega_2$ are two smooth hyperk\"ahler triples on $X$, and $\phi_0:\partial X\rightarrow\partial X$ is a diffeomorphism, such that $\bm\omega_1|_{\partial X}=\phi_0^*(\bm\omega_2|_{\partial X})$, then there exists a diffeomorphism $\phi:X\rightarrow X$,   $\phi|_{\partial X}=\phi_0$ such that  $\bm\omega_1=\phi^*\bm\omega_2$ on $X$.
\end{theorem}

\begin{remark} This theorem implies that
the map $p:\mathcal{M}^+/\mathcal{G}_X\rightarrow \mathcal{N}^+/\mathcal{G}_Y$ defined in subsection \ref{moduli space} is injective, provided that $\mathcal{M}^+$ is nonempty.
\end{remark}

\begin{proof}
By the previous proposition, there exists a collar neighborhood $U$ of $\partial X$ and a diffeomorphism $\phi_1: U\rightarrow V\subset X$ such that $g_{\bm\omega_1}=\phi_1^*g_{\bm\omega_2}$, $\phi_1|_{\partial X}=\phi_0$. Since $g_{\bm\omega_1},g_{\bm\omega_2}$ are real analytic, $\phi_1$ is real analytic.  Fix $p_0\in U$ and a small neighborhood $U_0$ of $p_0$ in $X$. For any $p\in X\backslash \partial X$, choose a path $x(t),0\leq t\leq 1$ such that $x(0)=p_0, x(1)=p, x(t)\in X\backslash\partial X$, then an analytic continuation of the isometry $\phi_1|_{U_0}$ along $x(t)$ gives rise to an isometry defined near $p$. We claim that if we have two paths and two analytic continuations, then they define the same germ at $p$. In fact, one only needs to show that for the closed path $y_0(t)$ formed by concatenating these two paths, the isometry near $p_0$ given by an analytic continuation of $\phi_1|_{U_0}$  along $ y_0(t)$ has the same germ as $\phi_1|_{U_0}$ at $p_0$. Since $\pi_1(X,\partial X)=0$, $y_0(t)$ can be homotoped to a path $y_1(t)$ contained in $U$ via paths $y_s(t)$ in $X\backslash\partial X$, such that $y_s(0)=y_s(1)=p_0$. Since $\phi_1:U\rightarrow V$ is a globally defined isometry, by uniqueness of analytic continuation, we know that any analytic continuation of $\phi_1|_{U_0}$  along $y_1(t)$ must coincide with $\phi_1$, which finishes the proof of the claim by invariance of analytic continuation via homotopy. This shows that $\phi_1:U\rightarrow V$ can be extended to a global isometry $\phi:X\rightarrow X$ by analytic continuation. Hence $\bm\omega_1=\phi^*\bm\omega_2$ by the same argument as before.
\end{proof}

Given the compactness result,  Theorem \ref{unique continuation hyperkahler}, together with the theorem of Ebin-Palais on   properness of diffeomorphism group action on the space of Riemannian metrics on a closed manifold (See \cite{ebin1970manifold}), one can answer the question raised at the beginning of this subsection, thus finish the proof of Theorem \ref{convergence of triples}, Theorem \ref{convergence of triples enhancements}:

Suppose we have two sequences of diffeomorphism $\phi_i,\psi_i$ of $X$ such that  $\phi_i^*\bm\omega_i\rightarrow \bm\omega$, $\psi_i^*\bm\omega_i\rightarrow\bm\omega'$, then $(\phi_i|_{\partial X})^*\bm\gamma_i\rightarrow \bm\omega|_{\partial X}$, $(\psi_i|_{\partial X})^*\bm\gamma_i\rightarrow \bm\omega'|_{\partial X}$, where $\bm\gamma_i=\bm\omega_i|_{\partial X}$. Since $\bm\gamma_i$ converges to   $\bm\gamma$ in Cheeger-Gromov sense, there exists diffeomorphisms $u_i:\partial X\rightarrow\partial X$ such that $u_i^*\bm\gamma_i\rightarrow\bm\gamma$. By 
the theorem of Ebin-Palais, we have for a subsequence $(\phi_i|_{\partial X})^{-1}\circ u_i , (\psi_i|_{\partial X})^{-1}\circ u_i $ converge to some diffeomorphisms $u,u'$ on $\partial X$, respectively(because their inverses converge). 
Hence we also have $u_i^*\bm\gamma_i\rightarrow u^*\bm\omega|_{\partial X}$, $u_i^*\bm\gamma_i\rightarrow (u')^*\bm\omega'|_{\partial X}$, so $\bm\omega|_{\partial X}=(u'\circ u^{-1})^*\bm\omega'|_{\partial X}$. Note that the positive mean curvature condition implies that $\pi_1(X,\partial X)=0$ (See Proposition \ref{diam}),  then by Theorem \ref{unique continuation hyperkahler}, there exists a diffeomorphism $\varphi$ on $X$ with $\bm\omega'=\varphi ^*\bm\omega$.

\subsection{Some discussions}

Suppose $X$ is a connected complete metric space, and there exists a finite set $\Sigma=\{p_1,\cdots,p_m\}, m\geq 0$ such that $X\backslash\Sigma$ is a smooth flat hyperk\"ahler
4-manifold with nonempty boundary, and each boundary component $Y_1,\cdots, Y_n, n\geq 1$ is isometric to flat $\mathbb{R}^3$,  $X\backslash \cup_{i=1}^n{Y_i}$ is a flat hyperk\"ahler 4-orbifold and $\Sigma$ is the set of all orbifold points. Our goal is to classify all such $X$.

The motivation of this problem is the following:
\begin{proposition}\label{classification motivation}
Let $(X_i,g_i)$ be a sequence of compact hyperk\"ahler 4-manifolds with boundary, such that $$i_{b,g_i}\geq i_0, inj_{\partial X_i}\rightarrow\infty,$$ $$|S_i|\rightarrow 0, |\nabla^{j+1}_{\partial X_i} H_i|\rightarrow 0, |\nabla^j_{\partial X_i} Rm_{\partial X_i}|\rightarrow 0$$ uniformly on $\partial X_i$ for all $j\geq 0$, and $$\int_{X_i}|Rm_{g_i}|^2\leq C.$$ Then for any $p_i\in X_i$ $d_{g_i}(p_i,\partial X_i)\leq K$, a subsequence of $(X_i,g_i,p_i)$ converges in pointed Gromov Hausdorff sense to a complete metric space $(X_\infty,d_\infty,p_\infty)$. The limit $(X_\infty,d_\infty)$ satisfies all properties of  $X$ above.
\end{proposition}

By Theorem \ref{good boundary}, the geometry near boundary is nicely controlled,  so this theorem is a consequence of \cite{anderson1989ricci} or \cite{bando1989construction} etc.

\begin{remark}
In Theorem \ref{main theorem hyperkahler metric}, under the mean positive condition, if we allow $-2$ curves in $X_i$, then we are unable to show that $i_{b,g_i}$ has a lower bound. In fact, the bad case is that focal points are exactly those curvature blow-up points, and they approach the boundary in a moderate rate. However, we can say something within the scale of the boundary injectivity radius. Suppose $i_{b,g_i}\rightarrow 0$, then we rescale the metric such that they are equal to 1, then the rescaled metric satisfies the assumption of Proposition \ref{classification motivation}.
\end{remark}

\begin{theorem}\label{classification}
$X$ is isometric to one of the following:
\begin{itemize}
    \item  $\text{(m=0, n=1)}$ $\mathbb{R}^4_+ $;
    \item  $\text{(m=0, n=2)}$   the region in $\mathbb{R}^4$ bounded by two parallel hyperplanes;
    \item $\text{(m=1, n=1)}$   the connected component of $0$ in   $(\mathbb{R}^4\backslash H)/{\mathbb{Z}_2}$, where $H$ is a hyperplane such that $0\notin H$. 
\end{itemize}

\end{theorem}

\begin{proof}

Consider another copy of $X$, glue them together along $Y_1,\cdots, Y_n$,  we get a complete flat hyperk\"ahler orbifold $\hat{X}=X\sqcup_{id} X, id:\cup_{i=1}^nY_i\rightarrow \cup_{i=1}^n Y_i\subset X$, which contains $Y_1,\cdots, Y_n$ as smooth hypersurfaces. Then $\hat X$ is a $(SU(2),\mathbb{R}^4)$ orbifold in the sense of \cite{thurston1979geometry}. Let  
$\tilde{X}$ be the universal covering orbifold of $\hat X$, then we have a developing map $D:\tilde {X}\rightarrow \mathbb{R}^4$, and since $\tilde{X}$ is a complete orbifold, $D$ is a covering map.  Hence $\tilde{X}$ is homoeomorphic to $\mathbb{R}^4$, and $\hat X$ is isometric to $\mathbb{R}^4/\Gamma$, where $\Gamma$ is a discrete subgroup of $\mathbb{R}^4\rtimes SU(2)$. Let $r$ be the projection map to the second factor, then $r(\Gamma)$ is a finite subgroup of $SU(2)$, and we have a short exact sequence
$$0\rightarrow \Gamma\cap\mathbb{R}^4\rightarrow \Gamma\rightarrow r(\Gamma)\rightarrow 0.$$ 
Then $\text{rank} (\Gamma\cap\mathbb{R}^4)\leq 1$, since otherwise $\hat{X}$ is a quotient of $\mathbb{R}^2\times T^2$ and cannot contain a flat $\mathbb{R}^3$ as a Riemannian submanifold. 

If $\Gamma\cap\mathbb{R}^4=0$, then $\hat{X}\cong \mathbb{R}^4/r(\Gamma)$ and
 hence $r(\Gamma)=\{1\}$, since otherwise $\hat{X}$ has exactly one orbifold point, contradiction. Hence, $m=0, n=1$ and  ${X}$ is isometric to $\mathbb{R}_+^4$.

If $\Gamma\cap\mathbb{R}^4=\mathbb{Z}a$ for some $0\neq a\in\mathbb{R}^4$, let $\pi: \mathbb{R}^4\rightarrow \mathbb{R}^4/\Gamma$ be the covering map, then $\pi^{-1}(Y_1)$  is a complete totally geodesic submanifold of $\mathbb{R}^4$, hence a countable disjoint union of parallel hyperplanes.
Pick one of them $Z$, then $\pi^{-1}(Y_1)$ is a disjoint union of $\gamma. Z$ for $\gamma\in \Gamma$. Suppose $Z$ is defined by $b^Tx+c=0$, then $\gamma^{-1}.Z$ is defined by $b^T r(\gamma)x+c'=0$. Since they are parallel to each other, $b^T=\pm b^Tr(\gamma)$, and hence $\pm 1$ is an eigenvalue of $r(\gamma)$, which forces $r(\gamma)=\pm 1$. Hence $r(\Gamma)=\{1\}$ or $r(\Gamma)=\mathbb{Z}_2$. If $r(\Gamma)=\{1\}$, then $\Gamma$ is generated by $x\mapsto x+a$, so $m=0, n=2$ and $X$ is isometric to the region in $\mathbb{R}^4$ bounded by two parallel hyperplanes; If $r(\Gamma)=\mathbb{Z}_2$, then $\Gamma$ is generated by $x\mapsto -x+d$ and $x\mapsto x+a$ for some $d\in\mathbb{R}$. Hence $m=1$ and $\mathbb{R}^4/\Gamma\backslash\{Y_1,Y_2,\cdots, Y_n\}$ has $n+1$ connected components. But from the gluing construction and that $X$ is connected, we know $\hat{X}\backslash\{Y_1,Y_2,\cdots, Y_n\}$ has exactly two connected components. Hence $n+1=2, n=1$, and $X$ is isometric to the last case in the conclusion.

\end{proof}

\section{Torsion-free hypersymplectic manifolds with boundary}\label{section torsion-free triples}

\subsection{Preliminaries}

Recall that a hypersymplectic triple $\bm\omega=(\omega_1,\omega_2,\omega_3)$ on an oriented 4-manifold $X$ with boundary is a definite triple of symplectic forms. Write
 $$\omega_i\wedge\omega_j=2Q_{ij}\mu$$
 Denote $\bm{Q}=(Q_{ij})$, $\bm{Q}^{-1}=(Q_{ij})$, and $g=g_{\bm\omega}$ the Riemannian metric. Recall that $\bm\omega$ is called torsion-free 
if for each $i$,
 \begin{equation}
      d(Q^{ij}\omega_j)=0,
 \end{equation} which is equivalent to
\begin{equation}\label{torsion-free condition}
    dQ^{ij}\wedge \omega_j=0.
\end{equation}

Let us begin with some arguments and results in \cite{fine2020report}. Let $\mathcal{P}$ denotes the set of symmetric positive-definite 3 by 3 matrices. There are two Riemannian metrics on $\mathcal{P}$: the first one is the Euclidean metric 

\begin{equation}\label{usual laplacian torsion-free}
    \langle A, B\rangle =Tr(AB),
\end{equation}
 and the second one is the symmetric space metric, which has non-positive sectional curvature
\begin{equation}\label{hat laplacian torsion-free}
     \langle A, B\rangle_Q=Tr(Q^{-1}A Q^{-1}B) 
\end{equation}
at each point $Q\in \mathcal{P}$.

Then $\bm{Q}$ can be regarded as a map $\bm{Q}:X\rightarrow \mathcal{P}$. Let $\Delta \bm Q, \hat{\Delta}\bm Q$ denote the harmonic Laplacian of this map with respect to (\ref{usual laplacian torsion-free}),(\ref{hat laplacian torsion-free}), respectively. Explicitly, their components are related by
\begin{equation}\label{torsion-free Laplacian operator}
    (\hat\Delta\bm Q)_{ij}=\Delta Q_{ij}-Q^{km}\langle dQ_{ik},dQ_{mj}   \rangle,
\end{equation}

For a hypersymplectic triple $\bm\omega$ , the calculations in \cite{fine2018hypersymplectic} showed that the torsion-free condition is equivalent to 
\begin{equation}\label{torsion-free equations}
\hat\Delta\bm Q=0, \ \text{Ric}= \frac{1}{4}\langle d\bm Q\otimes d\bm Q\rangle_{\bm Q},
\end{equation}
where $\langle d\bm Q\otimes d\bm Q \rangle_{\bm Q}(u,v)=\langle \nabla_u \bm Q,\nabla_v \bm Q\rangle_{\bm Q}$.
Hence if $\bm\omega$ is torsion-free, then $\bm Q$ is a harmonic map with respect to (\ref{hat laplacian torsion-free}) and  $\text{Ric}\geq 0$. Then the scalar curvature $R$ of $g$ is $$R=\frac{1}{4}|d\bm Q|_{\bm Q}^2\geq 0,$$ 
which is a multiple of the energy density of the harmonic map $\bm Q$.
Take the trace of (\ref{torsion-free Laplacian operator}), we get 
\begin{equation}\label{subharmonic TrQ}
     \Delta Tr \bm Q=Q^{pq}\langle dQ_{kp},dQ_{qk}\rangle\geq 0.
\end{equation}
Moreover, \cite{fine2020report} showed that the function $R$ satisfies the inequality 
\begin{equation}\label{R Delta R}
R\Delta R\geq \frac{1}{2}|\nabla R|^2+\frac{1}{2}R^3,
\end{equation}
hence a contradiction argument implies that everywhere
\begin{equation}\label{subharmonic scalar curvature}
    \Delta R\geq 0.
\end{equation}
Then they used standard geometric analysis arguments for inequality (\ref{R Delta R}) and the fact $\text{Ric}\geq 0$ to conclude

\begin{proposition}\label{fy result}
Suppose $B(p,r)\subset X$ has compact closure, and $\partial B(p,r)\neq \emptyset$, then 
$R(p)\leq \frac{32}{r^2}$. In particular, a complete torsion-free hypersymplectic 4-manifold is hyperk\"ahler.
\end{proposition}
Note that in the latter case, there exists a constant matrix $B$ in $SL(3,\mathbb{R})$ such that $\bm \omega B$ is a hyperk\"ahler triple, and the Riemannian metric defined by  $\bm \omega$ is the same as the one defined by $\bm\omega B$. 

\subsection{Compactness for the boundary value problem}
 
Now suppose $X$ is an oriented 4-manifold with compact boundary $Y=\partial X$, then through the boundary exponential map, a neighborhood $U$ of $Y$ is diffeomorphic to $Y\times [0,a)$. Let $t$ denote the distance function $d(\cdot, Y)$, i.e., the projection $Y\times [0,a)\rightarrow [0,a)$, then in $U$, $\bm\omega$ can be written as
 $$\bm\omega=-dt\wedge*_{Y_t}\bm\gamma_t+\bm\gamma_t.$$
 where $Y_t=Y\times \{t\}$ and $*_{Y_t}$ is the Hodge star operator of $g|_{Y_t}$.
 $\bm\omega$ being closed is equivalent to 
 \begin{equation}
     d_{Y_t}\bm\gamma_t=0,
 \end{equation}
 \begin{equation}\label{torsion-free normal tangential 0}
     \frac{\partial \bm\gamma_t}{\partial t}=-d_{Y_t}(*_{Y_t} \bm\gamma_t).
 \end{equation}
  (\ref{torsion-free condition})
 is equivalent to 
\begin{equation}\label{torsion-free normal tangential}
    d_{Y_t}Q^{ij}\wedge *_{Y_t}\gamma_{t,j}+\frac{\partial Q^{ij}}{\partial t}\gamma_{t,j}=0,
\end{equation}
 \begin{equation}
     d_{Y_t}Q^{ij}\wedge \gamma_{t,j}=0.
 \end{equation}
 Note that
 (\ref{torsion-free normal tangential}) is equivalent to
 \begin{equation}\label{torsion-free normal tangential 2}
     \frac{\partial Q^{ij}}{\partial t}=\frac{d_{Y_t}Q^{ik}\wedge \eta_{t,j}\wedge *_{Y_t}\gamma_k}{\text{vol}_{t}},
 \end{equation}
 where $\text{vol}_t=\eta_{t,1}\wedge \eta_{t,2}\wedge\eta_{t,3}$ and equals to the Riemannian volume form of $g|_{Y_t}$.  From calculations in Lemma \ref{hyperkahler triple second fundamental form} and (\ref{torsion-free normal tangential 0}), (\ref{torsion-free normal tangential 2}) , it is easy to see that the second fundamental form $II(e_i,e_j)$ is in algebraic terms of $\bm\eta, d_{\partial X}\bm\eta,\bm Q, d_{\partial X}\bm Q$.

 Now let us try to prove Theorem \ref{convergence of hypersymplectic triples}, starting with basic observations.
 In Theorem \ref{convergence of hypersymplectic triples}, suppose $\bm\omega_i|_{\partial X}, \bm Q_i$ converge in Cheeger-Gromov sense to the limit,
 then we have $\text{diam}(\partial X,g_i|_{\partial X})\leq C, \text{vol}_{g_i}(\partial X)\leq C, inj_{\partial X, g_i|_{\partial X}}\geq i_0$,  $|\nabla_{\partial X}^j Rm_{\partial X,g_i}|\leq C_j$, $|\nabla ^j_{\partial X} S_i|\leq C_j$. By (\ref{subharmonic scalar curvature}) and maximum principle, $R_i$  is uniformly bounded on $X$, so $|\text{Ric}_{g_i}|$ is uniformly bounded on $X$ since $\text{Ric}_{g_i}$ is non-negative and its trace is uniformly bounded.  Due to the uniform mean positive curvature condition, and $\text{Ric}_{ g_i}\geq 0$,  we have an upper bound of $\sup\limits_{p\in X}d_{g_i}(p,\partial X)$, $\text{vol}_{g_i}(X)$ by Proposition \ref{diam}, and in particular an upper bound of $\text{diam}(X, g_i)$. The Chern-Gauss-Bonnet formula thus gives an upper bound of $\int_{X}|Rm_{g_i}|^2$. Also, by (\ref{subharmonic TrQ}) and maximum principle, $Tr \bm Q_i$ is uniformly bounded on $X$, so $\bm Q_i$ is uniformly bounded on $X$ and then $
|d\bm Q_i|$ is uniformly bounded on $X$. 

 Given these conclusions, we need to verify that all propositions that were used to prove Theorem \ref{convergence of triples enhancements} adapt to the torsion-free hypersymplectic setting. Firstly, we digress to discuss elliptic regularity for torsion-free equations (\ref{torsion-free equations}). In harmonic coordinates in $B_2$ or $B_2^+$, view $g$ as a 4 by 4 matrix of functions, then we have a system of PDEs in $g, Q$: 
 \begin{equation}\label{torsion-free equation 1 harmonic}
     \Delta_g Q_{kl}-Q^{pq}\langle dQ_{kp},dQ_{ql}\rangle_g=0,
 \end{equation}
 \begin{equation}\label{torsion-free equation 2 harmonic}
     \Delta_g g_{ij}+B_{ij}(g,\partial g)=-\frac{1}{2}Q^{ab}Q^{cd}{\partial_i}Q_{bc}{\partial_j}Q_{da}.
 \end{equation}
 From this, by a boostrapping argument, one sees interior regularity: fix $\beta\in (0,1)$. If $\bm Q$ is $C^1$ bounded, $g$ is $C^{1,\beta}$ bounded, and they are uniform positive on $B_2$, then all derivatives of $\bm Q$ and $g$ are bounded. For boundary regularity, there is no techinical difficulty to get the following version from Neumann boundary conditions (\ref{Neumann boundary condition 1}),(\ref{Neumann boundary condition 2}): if all tangential
 derivatives of $\bm Q$, $g$, $H$ are bounded on $\tilde{B}_2$, and  $\bm Q$ is $C^1$ bounded, $g$ is $C^{1,\beta}$ bounded, and they are uniformly positive on $B_2^+$, then all derivatives of $\bm Q$ and $g$ are bounded in $B_1^+$. If in both cases, we also assume $g$ is $C^{k,\beta}$ close to identity in $B_2$ or $B_2^+$, then similarly, $g$ is $C^{k,\alpha}$ close to identity for any $\alpha\in (\beta,1)$.

Following the proof of Theorem \ref{local harmonic radius lower bound}, we get the following two propositions.

\begin{proposition}\label{harmonic radius and bounds for Q} Let $(X,g,\bm\omega)$ be a  torsion-free hypersymplectic manifold and $B(p,r)$ is a metric ball that has compact closure, $\partial B(p,r)\neq \emptyset$. Suppose  for any $q\in B(p,r)$, $$inj_q \geq cd(q,\partial B(p,r)), $$
$$Tr\bm Q, R\leq  C.$$
Fix $\Lambda>1,0<\alpha<1$, then for any $k\geq 0$,  $q\in B(p,r),$  $$r_h^{k,\alpha}(q,g, \Lambda)\geq C_k'd(q,\partial B(p,r)).$$
In particular, $|\nabla^k \bm Q|\leq C_k''$ in $B(p,\frac{r}{2})$.

\end{proposition}

\begin{proposition} Let $(X,g,\bm\omega)$ be a  compact torsion-free hypersymplectic manifold with boundary. Suppose $i_b\geq i_0$, $inj_{X}\geq i_0$, $inj_{\partial X}\geq i_0$, $Tr \bm Q$, $R\leq C$ on $\partial X$ and $|\nabla^j_{\partial X} Rm_{\partial X}|\leq C_j$, $|\nabla^j_{\partial X} S|\leq C_j $, $|\nabla^j_{\partial X} \bm Q|\leq C_j$ on $\partial X$, $\forall k\geq 0$. Fix $\Lambda>1, 0<\alpha<1$, then for any $k\geq 0$, $q \in X$,
$$r_h^{k,\alpha}(q,g, \Lambda)\geq C_k'''.$$
\end{proposition}
 
 Then we have 
 
\begin{proposition}
Proposition \ref{no bubble 2} holds for torsion-free hypersymplectic manifolds $(X,g,\bm\omega)$, provided an upper bound of $Tr\bm Q$ in $B(p,5)$.
\end{proposition}

\begin{proof}

The proof is almost the same as there. Let us list the ingredients here: 

\begin{itemize}
\item We have Bishop-Gromov volume comparison, since $\text{Ric}_{g_i}\geq 0,$	
\item Before rescaling, $ R_i, |\text{Ric}_{g_i}|$ are automatically bounded by Proposition \ref{fy result}.
\item 
  For the rescaled metric $\tilde{g}_i$, the 
  curvature bound and the volume non-collapsing condition imply injectivity radius lower bound on compact sets, hence by Proposition \ref{harmonic radius and bounds for Q}, we have harmonic radius lower bounds as well as bounds for derivatives of $\bm Q_i$ on compact sets, so we have pointed Cheeger-Gromov convergence of a subsequence $(M_i,g_i,\bm\omega_i, q_i)$.

\item The limit $\bm\omega_\infty$ is a hyperk\"ahler triple up to a $SL(3,\mathbb{R})$ rotation, because $g_\infty$ is scalar flat, or because of Proposition \ref{fy result}.

\end{itemize}
So, we get the contradiction in the same way. 
\end{proof}

With the above three Propositions as tools and $\epsilon$-regularity, argue the same way as in Theorem \ref{good boundary}, one gets an analogous version of Theorem \ref{good boundary}, i.e., curvature control within  $i_b$, assuming $i_b\geq i_0$.  

\begin{remark}
We make a remark about the proof of $\epsilon$-regularity for torsion-free hypersymplectic manifolds here. Firstly, the proof in \cite{sun2021collapsing} Theorem 3.21 directly applies to this case by using Proposition \ref{fy result}. Alternatively, we can apply Remark 8.22 in \cite{cheeger2006curvature} to conclude that $g$ has $C^{1,\alpha}$ bounded covering geometry when curvature $L^2$ norm is small. By (\ref{torsion-free equations}), we have $|\nabla \text{Ric}|\leq C$, hence $g$ has $C^{2,\alpha}$ bounded covering geometry and $|Rm|$ is bounded.
\end{remark}

Now one can finish the proof of compactness part of  Theorem \ref{convergence of hypersymplectic triples} by the same arguments in Section \ref{section main proof}. Note that $\text{Ric}\geq 0, H>0$ is enough for the focal point argument.

\subsection{Uniqueness}

Finally, we prove the uniqueness part of Theorem \ref{convergence of hypersymplectic triples}.

\begin{proposition} Let $\bm\omega_1,\bm\omega_2$ be two torsion-free hypersymplectic triples on an oriented 4-manifold $X$ with compact boundary $Y=\partial X
$.
Suppose $\bm\gamma_1=\bm\gamma_2, \bm Q_1=\bm Q_2$ on $\partial X$, then $\bm\omega_1=\bm\omega_2$ in geodesic gauges of $g_{\bm\omega_i}$ near $\partial X$.
\end{proposition}

\begin{proof}
$\bm\omega_i$ defines a torsion-free $G_2$ structure $\phi_i$ on $X\times T^3$ via (\ref{g2 and hypersymplectic}), which defines a warpped product metric 
\begin{equation}\label{warpped product metric}
    g_{\phi_i}=g_{\bm\omega_i}+ Q_{ij}dt^idt^j.
\end{equation}
In the geodesic gauge of $g_{\bm\omega_i}$, write
$$\bm\omega_i=-dt\wedge *_{Y_t}\bm\gamma_t+\bm\gamma_t,$$
where $t$ is the distance function $d_{g_{\bm\omega_i}}(\cdot,\partial X)$.
Hence 
\begin{equation}\label{g2 geodesic gauge}
\phi_i=-dt\wedge \theta_{t,i}+\rho_{t,i},
\end{equation}
where $$\rho_{t,i}=dt^1\wedge dt^2\wedge dt^3-\gamma_i^1\wedge dt^1-\gamma_i^2\wedge dt^2-\gamma_i^3\wedge dt^3,$$
$$\theta_{t,i}=-*_{Y_t}\gamma_{t,i}^1\wedge dt^1-*_{Y_t}\gamma_{t,i}^2\wedge dt^2-*_{Y_t}\gamma_{t,i}^3\wedge dt^3$$
By (\ref{warpped product metric}), $t$ can also be viewed as $d_{g_{\phi_i}}(\cdot,\partial X\times T^3)$, so (\ref{g2 geodesic gauge}) is written in the geodesic gauge of $g_{\phi_i}$.
By the calculations in \cite{donaldson2018remarks} Section 2.2, for $i=1,2$, both $g_{\phi_i}|_{\partial X\times T^3}$ and the second fundamental forms of $\partial X\times T^3$ are equal to each other, since they are explicitly in terms of $\theta_{0.i},\rho_{0,i}$, which are in terms of $\bm\gamma_i,\bm Q_i$. Since $g_{\phi_i}$ are Ricci-flat, by \cite{biquard:hal-02928859} Theorem 4, $g_{\phi_1}=g_{\phi_2}$. Since $\nabla^{g_{\phi_1}}|\phi_1-\phi_2|^2=0$ and $\phi_1-\phi_2=0$ at one point, we have $\phi_1=\phi_2$, $\bm\omega_1=\bm\omega_2$.

\end{proof}

Note that for a torsion-free hypersymplectic triple $\bm\omega$, the metric $g_{\bm\omega}$ is real analytic with respect to the analytic structure defined by harmonic coordinates, due to elliptic regularity of (\ref{torsion-free equation 1 harmonic})(\ref{torsion-free equation 2 harmonic}), so the arguments in subsection \ref{uniqueness 1} shows global uniqueness:

\begin{theorem}\label{unique continuation torsion-free hypersymplectic}
Let $X$ be a connected 4-manifold with boundary, $\pi_1(X,\partial X)=0$. Suppose $\bm\omega_1,\bm\omega_2$ are two smooth torsion-free hypersymplectic triples on $X$, and $\varphi_0:\partial X\rightarrow\partial X$ is a diffeomorphism, such that $\bm\omega_1|_{\partial X}=\varphi_0^*(\bm\omega_2|_{\partial X})$, $\bm Q_1|_{\partial X}=\varphi_0^*\bm Q_2|_{\partial X}$, then there exists a diffeomorphism $\varphi:X\rightarrow X$,   $\varphi|_{\partial X}=\varphi_0$,  such that $\bm\omega_1=\varphi^*\bm\omega_2$ on $X$.
\end{theorem}

\bibliographystyle{plain}

\end{document}